\documentclass[10pt,twoside]{amsart}
\usepackage{amsmath}
\usepackage{amssymb}
\usepackage{amsthm}
\usepackage{latexsym}
\usepackage{amscd}
\usepackage{xypic}
\xyoption{curve}
\usepackage{ifthen}
\usepackage{hyperref}
\usepackage{graphicx}

\usepackage{xcolor}

\newtheorem{theorem}{Theorem}[section]
\newtheorem{lemma}[theorem]{Lemma}
\newtheorem{proposition}[theorem]{Proposition}
\newtheorem{corollary}[theorem]{Corollary}

\theoremstyle{definition}
\newtheorem{remark}[theorem]{Remark}
\newtheorem{definition}[theorem]{Definition}
\newtheorem{example}[theorem]{Example}

\newcommand {\spec}{\mathrm{spec}}

\newcommand {\Hom}{\mathrm{Hom}}
\newcommand {\Ext}{\mathrm{Ext}}

\newcommand {\rk}{\mathrm{rk}}

\newcommand {\Hilb}{\mathcal{H}\kern -0.25ex{\mathit ilb\/}}

\newcommand {\cA}{\mathcal{A}}
\newcommand {\cB}{\mathcal{B}}
\newcommand {\cP}{\mathcal{P}}

\newcommand {\bZ}{\mathbb{Z}}

\newcommand {\bP}{\mathbb{P}}

\newcommand{\cE}{{\mathcal E}}
\newcommand{\cK}{{\mathcal K}}
\newcommand{\cF}{{\mathcal F}}
\newcommand{\cG}{{\mathcal G}}
\newcommand{\cM}{{\mathcal M}}
\newcommand{\cN}{{\mathcal N}}
\newcommand{\cO}{{\mathcal O}}
\newcommand{\cT}{{\mathcal T}}
\newcommand{\cI}{{\mathcal I}}
\newcommand{\calL}{{\mathcal L}}

\newcommand{\Pic}{\operatorname{Pic}}

\def\p#1{{\bP^{#1}}}
\def\mapright#1{\mathbin{\smash{\mathop{\longrightarrow}
\limits^{#1}}}}

\def\ga#1{{{\accent"12 #1}}}

\title[rank two aCM bundles]{Rank two aCM bundles \\
on the del Pezzo threefold of degree $7$}



\author[G. Casnati, M. Filip, F. Malaspina]{Gianfranco Casnati, Matej Filip, Francesco Malaspina}
\thanks{The first and third authors are members of the GNSAGA group of INdAM and are supported by the framework of PRIN 2010/11 \lq Geometria delle variet\ga a algebriche\rq, cofinanced by MIUR. 
The second author would like to thank G. Casnati and F. Malaspina for their hospitality during his stay at Politecnico di Torino in May 2014}

\date{\today}

\begin{document}

\begin{abstract}
We classify indecomposable aCM bundles of rank $2$ on the del Pezzo threefold of degree $7$ and analyze the corresponding moduli spaces.
\end{abstract}

\maketitle

\section{Introduction}
Let $\p N$ be the projective space of dimension $N$ over an algebraically closed field of characteristic $0$ and let $X\subseteq\p N$ be an $n$--dimensional smooth projective variety (i.e. an integral connected closed subscheme). 

The variety $X$ is polarized by $\cO_X(h):=\cO_{\p N}(1)\otimes\cO_X$ and we will assume that it is {\sl arithmetically Cohen--Macaulay} (aCM for short), i.e. $h^i\big(X,\cI_{X\vert\p N}(th)\big)=0$ for $i=1,\dots, n$ and $t\in\bZ$. 

Such a vanishing obviously implies $h^i\big(X,\cO_X(th)\big)=0$ for $i=1,\dots, n-1$ and $t\in\bZ$ and it is interesting to inspect the existence of other vector bundles on $F$ with the same property. The first result in this direction is Horrocks theorem which asserts that this is not the case when $X=\p N$ (see \cite{O--S--S} and the references therein). 

The property of being aCM is invariant up to shifting degrees. For this reason we restrict our attention to {\sl initialized} bundles, i.e. bundles $\cE$ such that $h^0\big(X,\cE(-h)\big)=0$ and $h^0\big(X,\cE\big)\ne0$. Moreover, we are also interested to {\sl indecomposable} bundles, i.e. bundles which do not split in a direct sum of bundles of lower rank.

D. Eisenbud and D. Herzog in \cite{E--He} proved that the list of aCM varieties supporting at most finitely many aCM bundles is very short. Indeed they are linear spaces, smooth quadrics, rational normal curves, the Veronese surface, up to three reduced points and the smooth, rational, scroll of dimension $n=2$ and degree $d=3$. T. Drozd and G.M. Greuel (see \cite{D--G}) call these varieties of {\sl finite representation type}.

On the opposite side M. Casanellas and R. Hartshorne proved in \cite{C--H1} and \cite{C--H2} that a smooth cubic surface is endowed with families of arbitrary dimension of non--isomorphic, indecomposable aCM bundles satisfying an extra technical hypothesis, namely their associated module of section has the maximal number $rd$ of minimal generators: we call such bundles {\sl Ulrich}. 

Again following \cite{D--G} we  briefly say that the smooth cubic in $\p3$ is of {\sl wild representation type}. The study of the wildness of a variety has been a fruitful field of research: without any claim of completeness we recall \cite{C--MR}, \cite{F--M}, \cite{MR--PL1}, \cite{MR--PL2}, \cite{PL--T}.

Very recently D. Faenzi and J. Pons--Lopis proved in \cite{F--PL} that all the aCM varieties of positive dimension which are not cones are of wild representation type except 
\begin{itemize} 
\item the aforementioned varieties of finite representation type,
\item elliptic curves, nodal rational curves and every smooth, rational, scrolls of dimension $n=2$ and degree $d=4$ (these are called varieties of {\sl tame representation type} in \cite{D--G}).
\end{itemize}

Nevertheless, the geometric description and classification of aCM bundles on other varieties is of great interest by itself. The first step is to inspect the cases of low ranks $r=1,2$. Many results are known when the Picard group of $X$ is generated by $\cO_X(h)$ (see  \cite{A--C}, \cite{Ca}, \cite{C--M2}, \cite{Fi}, \cite{K--R--R}, \cite{Ma1}, \cite{Ma2} for some non--exhaustive examples). Indeed, in this case, aCM line bundles correspond to aCM divisors on $X$, while aCM vector bundles of rank $2$ correspond to subschemes of pure codimension $2$ in $X$  via the so called Serre correspondence. 

When the Picard group is no more generated by $\cO_X(h)$ the picture becomes quite complicated. Indeed the aforementioned correspondences of aCM bundles of low rank with subschemes of small codimension in $X$ is not so immediate. As examples of this fact we recall \cite{Fa}, where a complete description of aCM bundles of ranks $r=1,2$ on a smooth cubic in $\p3$ is presented, and \cite{Wa}, where the author gives the classification of aCM line bundles on a smooth quartic in $\p3$. In \cite{C--N} examples of indecomposable, initialized, aCM bundles of rank $2$ whose general sections vanish on a divisor are given.

From this point of view, the contribution of E. Arrondo and L. Costa in \cite{A--C} shed some light on another class of important varieties, namely  {\sl del Pezzo variety} with $n\ge3$, i.e. a varieties $X$ such that $\omega_X\cong\cO_X((1-n)h)$. 

These varieties are completely classified (see \cite{I--P}) and their degrees satisfy $3\le \deg(X)=N-1\le 8$. If $ \deg(X)=3,4,5$, then $\Pic(X)$ is generated by $\cO_X(h)$ and the description of indecomposable, initialized aCM bundles of rank $2$ on them is in \cite{A--C}. When $d=8$, then $\Pic(X)$ is still principal but not generated by $\cO_X(h)$: the classification of indecomposable, initialized aCM bundles of rank $2$ on $X$ can be found in \cite{Ca} as a part of a more general study of del Pezzo varieties.

In the papers \cite{C--F--M1}, \cite{C--F--M2}, \cite{C--F--M3} the authors deal with the case $\deg(X)=6$. In this case $\Pic(X)$ is free, but its rank can be either $2$, or $3$. 

In the present paper we will describe the remaining case, namely the case of the del Pezzo variety $F$ of degree $7$. In this case $F$ is unique up to isomorphisms. Indeed it can be obtained as the blow--up of $\p3$ at a single point $P$, thus we have a natural morphism $\sigma\colon F\to \p3$. 

Moreover there is also another important identification for $F$. Indeed, if $\cP:=\cO_{\p2}\oplus\cO_{\p2}(1)$ then $F\cong\bP(\cP)$. Thus we also have a canonical projection $\pi\colon F\cong\bP(\cP)\to\p2$. 

We denote by $\xi$ and $f$ the classes of the bundles $\cO_{\bP(\cP)}(1)$ and $\pi^*\cO_{\p2}(1)$ respectively. The embedding $F\subseteq\p8$ is induced by the linear system $\cO_F(h)=\cO_F(\xi+f)$ and $\omega_F\cong\cO_F(-2h)$.

We can now finally state the first main result of this paper.

\medbreak
\noindent
{\bf Theorem A.}
{\it If $\cE$ is an indecomposable, initialized, aCM bundle of rank $2$ on $F$, then the zero locus $E$ of a general section $s\in H^0\big(F,\cE\big)$ has pure codimension $2$. 

Moreover one of the following cases occurs.
\begin{enumerate}
\item $c_1(\cE)=0$ and $c_2(\cE)=f^2$: $\cE\cong p^*\cF$, where $\cF$ is a normalized bundle of rank $2$ on $\p2$ with $c_2(\cF)=1$, and $E$ is a line.
\item $c_1(\cE)=0$ and $c_2(\cE)=\xi^2-f^2$: in this case $\cE$ fits into an exact sequence of the form
$$
0\longrightarrow \cO_F(f-\xi)\longrightarrow\cE\longrightarrow\cO_F(\xi-f)\longrightarrow0,
$$
it is not the pull--back of a bundle via either $\sigma$, or $\pi$ and $E$ is a line.
\item $c_1(\cE)=f$ and $c_2(\cE)=f^2$:  $\cE\cong (p^*\Omega_{\p2}^1)(2f)$ and $E$ is a line.
\item $c_1(\cE)=2\xi+f$ and $c_2(\cE)=2\xi^2+f^2$: $\cE\cong (p^*\Omega_{\p2}^1)(\xi+2f)$ and $E$ is a rational quintic curve in $\p8$.
\item $c_1(\cE)=2\xi+2f$ and $c_2(\cE)=3\xi^2+3f^2$: $\cE\cong (p^*\cF)(h)$, where $\cF$ is a stable normalized bundle of rank $2$ on $\p2$ with $c_2(\cF)=2$, and $E$ is an elliptic normal curve in $\p8$.
\item $c_1(\cE)=2\xi+2f$ and $c_2(\cE)=4\xi^2+f^2$: $\cE\cong (\sigma^*\cG)(h)$, where $\cG$ is a null-correlation bundle on $\p3$, and $E$ is an elliptic normal curve in $\p8$.
\end{enumerate}
}
\medbreak

In particular we are able to identify Ulrich bundles among aCM ones, as those bundles corresponding to elliptic curves. The second result of the paper regards the (semi)stability of the bundles listed in Theorem A and, when it is possible, the description of their moduli spaces.

\medbreak
\noindent
{\bf Theorem B.}
{\it Let $\cE$ is an indecomposable, initialized, aCM bundle of rank $2$ on $F$. Then the following assertions hold.
\begin{enumerate}
\item If $c_1(\cE)=2\xi+2f$, then $\cE$ is $\mu$--stable. Once $c_2(\cE)$ is fixed (either $3\xi^2+3f^2$, or $4\xi^2+f^2$) the corresponding moduli spaces are smooth, irreducible and rational of dimension $5$.
\item If $c_1(\cE)=0$, then $\cE$ is strictly $\mu$--semistable.
\end{enumerate}
}
\medbreak

We already pointed out that $F$ is of wild representation type (e.g. from \cite{F--PL}), i.e. supports large families of aCM bundles.  A natural question is to inspect if it is also of {\sl Ulrich--wild representation type}, i.e. if it supports families of dimension $p$ of non--isomorphic indecomposable Ulrich bundles for arbitrary large $p$. 

Besides the aforementioned case of the smooth cubic surface treated in \cite{C--H2}, some other scattered results on Ulrich--wildness are known: e.g. see \cite{C--MR--PL} for Segre varieties but $\p1\times\p1$ and \cite{MR--PL2} for del Pezzo surfaces of degree up to $8$.

As a by--product of our classification, using a powerful result from \cite{F--PL}, we are able to prove the following result.

\medbreak
\noindent
{\bf Theorem C.}
{\it The variety $F$ is of Ulrich--wild representation type.}
\medbreak

The proofs of the above results are strictly interlaced and we prove them in several steps. For reader's benefit we quickly explain them below.

Firstly, we need to obtain bounds on the first Chern class $c_1(\cE)$ of an indecomposable initialized aCM bundle $\cE$ of rank $2$ on $F$. In Section \ref{sFirstProperties}, after recalling some basic facts about aCM bundles, we prove a first rough upper bound on $c_1(\cE)$ (see Lemma \ref{lGG}), showing also that the zero--locus $(s)_0$ of a general section $s$ of $\cE$ is not empty.   

If $D$ is the divisorial part of $(s)_0$, then the effectiveness of $c_1(\cE)-D$ proved in Section \ref{sEffective} yields a lower bound of $c_1(\cE)$ in Section \ref{sTheoremAB}. We collect all possible values of $c_1(\cE)$ and $c_2(\cE)$ (see Table B), obtaining that $(s)_0$ has pure codimension $2$ (see Theorem \ref{tDzero}). Theorems A and B are proved in the last two sections. In Section \ref{sUlrich} we analyze bundles that correspond to elliptic curves and thus we prove assertions (5) and (6) of Theorem A and (1) of Theorem B. Moreover we also prove Theorem C in Paragraph \ref{sWild}. The remaining parts of Theorems A and B are proven in Section \ref{sNonUlrich}.   

Throughout the whole paper we refer to \cite{I--P} and \cite{Ha2} for all the unmentioned definitions, notation and results.

\section{Some facts on aCM locally free sheaves}
\label{saCM}
Throughout the paper $k$ will denote an algebraically closed field of characteristic $0$. 
In what follows  $F\subseteq\p8$ will denote the unique del Pezzo threefold of degree $7$. We set $\cO_F(h):=\cO_{\p 8}(1)\otimes\cO_F$. We recall that $h^i\big(\p8,\cI_{F\vert \p3}(t)\big)=0$ for $i=1,2,3$ and every $t\in\bZ$.

As pointed out in the introduction $F$ is identified with the blow up of $\p3$ at a point $P$ and it is endowed with two natural morphisms $\sigma\colon F\to \p3$ and $\pi\colon F\cong\bP(\cP)\to\p2$, where $\cP:=\cO_{\p2}\oplus\cO_{\p2}(1)$. Recall that $\sigma^{-1}(P)\cong\p2$ and $\sigma$ induces an isomorphism $F\setminus \sigma^{-1}(P)\cong\p3\setminus\{\ P\ \}$.

We denote by  $\xi$ and $f$ the classes of the bundles $\cO_{\bP(\cP)}(1)$ and $\pi^*\cO_{\p2}(1)$ respectively.  Trivially $\pi^*\cO_{\p2}(1)$ is globally generated. Since $\cP$ is globally generated, it follows that the same holds for $\cO_{\bP(\cP)}(1)$. The embedding $F\subseteq\p8$ is induced by the linear system $\cO_F(h)=\cO_F(\xi+f)$ and $\omega_F\cong\cO_F(-2h)$.

We have an isomorphism
$$
A(F)\cong\bZ[\xi,f]/(f^3,\xi^2-\xi f),
$$
hence $\xi^3=\xi^2 f=\xi f^2$ are the class of a point. 

We recall below some definitions quickly given in the introduction.

\begin{definition}
  Let $\cE$ be a vector bundle on $F$.
  \begin{itemize}
\item The bundle $\cE$ is aCM if  $  H^i_*\big(F,\cE\big)=0$ for each $i=1,2$.
\item The bundle $\cE$ is initialized if  $  \min\{\ t\in\bZ\ \vert\ h^0\big(F,\cE(th)\big)\ne0\ \}=0$.
\item The bundle $\cE$ is {\sl Ulrich} if it is initialized, aCM and $h^0\big(F,\cE\big)=7\rk(\cE)$. 
\end{itemize}
\end{definition}

On the one hand, if $\cE$ is Ulrich, then $m(\cE)=7\rk(\cE)$. On the other hand the generators of
$H^0\big(F,\cE\big)$ are minimal generators of $H^0_*\big(F,\cE\big)$
because $\cE$ is initialized. Thus $\cE$ is globally generated. 

If $s$ is a global section of  a rank $2$ vector bundle $\cE$, its zero--locus
$(s)_0\subseteq F$ is either empty or its codimension is at most
$2$ (e.g. if $\cE$ is globally generated, thanks to Bertini's theorem for the sections of a vector bundle). In the second case we write $(s)_0=E\cup D$
where $E$ has pure codimension $2$ (or it is empty) and $D$ has pure codimension $1$ (or it is empty). In particular $\cE(-D)$ has a section  vanishing
on $E$, which is thus locally complete intersection inside $E$. The corresponding Koszul complex induces the exact sequence
\begin{equation}
  \label{seqIdeal}
  0\longrightarrow \cO_F(D)\longrightarrow \cE\longrightarrow \cI_{E\vert F}(c_1(\cE)-D)\longrightarrow 0.
\end{equation}
It is well--known that the class of $E$ in $A(F)$ is $c_2(\cE(-D))=c_2-D(c_1-D)$ and its degree $hc_2(\cE(-D))$ (see \cite{Fu}, Theorem 14.4 b), Example 14.4.1). 

Twisting the above sequence by $\cO_F(-D)$ and restricting it to $E$ we obtain the following identification for the normal bundle of $E$ in $F$
\begin{equation}
  \label{NormalBundle}
\cE\otimes\cO_E\cong \cN_{E\vert F}.
\end{equation}
Moreover we also have the exact sequence
\begin{equation}
  \label{seqStandard}
  0\longrightarrow \cI_{E\vert F}\longrightarrow \cO_F\longrightarrow \cO_E\longrightarrow 0.
\end{equation}

The above construction can be reversed under suitable hypothesis: this is the Hartshorne--Serre correspondence (see \cite{Vo}, \cite{Ha1}, \cite{Ar}).

\begin{theorem}
  \label{tSerre}
Let $F$ be a smooth projective variety of dimension $n\ge2$, $E\subseteq F$ a locally complete intersection subscheme of codimension $2$ with $\det(\cN_{E\vert F})\cong{\mathcal L}\otimes\cO_E$ for a suitable $\mathcal L\in \Pic(F)$. 

If $h^2\big(F,{\mathcal L}^\vee\big)=0$, then there exists a vector bundle $\cE$ of rank $2$ on $F$ with $\det(\cE)={\mathcal L}$ and having a section $s$ such that $E=(s)_0$. Moreover, if also $h^1\big(F,{\mathcal L}^\vee\big)= 0$, such an $\cE$ is uniquely determined up to isomorphism by the above properties.
\end{theorem}

We recall some easy facts about the cohomology of line bundles on $F$ in the following Proposition (its proof is an immediate application of \cite{Ha2}, Exercise III.8.4 (a)).

\begin{proposition}
\label{pLineBundle}
Let ${\mathcal L}\cong\cO_F(\lambda_1\xi+\lambda_2f)$. Then 
$$
\begin{aligned}
h^0\big(F,{\mathcal L}\big)&=\sum_{j=0}^{\lambda_1}{\lambda_1+\lambda_2-j+2\choose2},\\
h^1\big(F,{\mathcal L}\big)&=\sum_{j=0}^{-\lambda_1-2}{\lambda_1+\lambda_2+j+3\choose2},\\
h^2\big(F,{\mathcal L}\big)&=\sum_{j=0}^{\lambda_1}{-\lambda_1-\lambda_2+j-1\choose2},\\
h^3\big(F,{\mathcal L}\big)&=\sum_{j=0}^{-\lambda_1-2}{-\lambda_1-\lambda_2-j-2\choose2}
 \end{aligned}
 $$
 where the summation is $0$ if the upper limit is less than the lower limit.
 
 Moreover, $\mathcal L$ is globally generated if and only if $\lambda_1, \lambda_2\ge0$.
\end{proposition}

The following corollaries follow trivially from the above proposition.

\begin{corollary}
\label{cNonVanishing}
Let ${\mathcal L}\cong\cO_F(\lambda_1\xi+\lambda_2f)$. Then 
$$
\begin{aligned}
h^0\big(F,{\mathcal L}\big)\ne0\qquad&\Leftrightarrow\qquad \lambda_1\ge0,\ \lambda_1+\lambda_2\ge0,\\
h^1\big(F,{\mathcal L}\big)\ne0\qquad&\Leftrightarrow\qquad \lambda_1\le -2,\ \lambda_2\ge1,\\
h^2\big(F,{\mathcal L}\big)\ne0\qquad&\Leftrightarrow\qquad \lambda_1\ge 0,\ \lambda_2\le-3,\\
h^3\big(F,{\mathcal L}\big)\ne0\qquad&\Leftrightarrow\qquad \lambda_1\le-2,\ \lambda_1+\lambda_2\le-4.
\end{aligned}
$$
\end{corollary}

\begin{corollary}
\label{cLineBundle}
The initialized, aCM, line bundle on $F$ are
$$
\cO_F,\qquad\cO_F(f),\qquad \cO_F(2f),\qquad \cO_F(\xi),\qquad \cO_F(\xi-f).
$$
There are no Ulrich line bundles on $F$.
\end{corollary}

In order to explain the important consequence of the above corollary for rank $2$ Ulrich bundles, we recall some well known facts. Let $X\subseteq\p N$ be a smooth projective variety of dimension $n$ and consider a vector bundle $\cF$ on it. The slope $\mu(\cF)$ and the reduced Hilbert polynomial $p_{\cF}(t)$ of $\cF$ are
$$
\mu(\cF)= c_1(\cF)H^{n-1}/\rk(\cF), \qquad p_{\cF}(t)=\chi(\cF(t))/\rk(\cF),
$$
where $H$ is a general hyperplane section in $X$. The bundle $\cF$ is $\mu$--semistable (with respect to $H$) if for all torsion--free quotient sheaves $\cK$ with $0<\rk(\cK)<\rk(\cF)$ we have
$\mu(\cK) \ge \mu(\cF)$, and $\mu$--stable if equality cannot hold. The bundle $\cF$ is called semistable (or, more precisely,
Gieseker--semistable with respect to $H$) if for all $\cK$ as above we have 
$p_{\cK}(t) \ge  p_{\cF}(t)$, and (Gieseker) stable again if  the equality cannot hold. We have the following chain of implications for $\cF$:
$$
\text{$\cF$ is $\mu$--stable}\ \Rightarrow\ \text{$\cF$ is stable}\ \Rightarrow\ \text{$\cF$ is semistable}\ \Rightarrow\ \text{$\cF$ is $\mu$--semistable.}
$$

Let $X\subseteq \p N$ be also aCM. If $\cF$ is a vector bundle of rank $r$, then there exists the coarse moduli space $\cM_X^{s}(\chi)$ parameterizing isomorphism classes of stable rank $r$ bundles on $X$ with Hilbert polynomial $\chi(t):=\chi(\cF(t))$. 

The scheme $\cM_X^{s}(\chi)$ is the disjoint union of open and closed subsets $\cM_X^{s}(r;c_1,\dots,c_r)$ whose points represent isomorphism classes of stable rank $r$ bundles with fixed Chern classes $c_i\in A^i(X)$.

By semicontinuity one can define the open locus $\cM_X^{s,aCM}(\chi)(r;c_1,\dots,c_r)\subseteq \cM_X^{s}(\chi)(r;c_1,\dots,c_r)$ parameterizing stable aCM bundles of rank $r$ on $F$ with Chern classes $c_1,\dots,c_r$.
For the following proposition see \cite{C--H2}.

\begin{proposition}
\label{pModUlrich}
There exist coarse moduli spaces $\cM_X^{s,U}(r;c_1,\dots,c_r)$ for stable Ulrich bundles of rank $r$ on $X$ with Chern classes $c_1,\dots,c_r$.
\end{proposition}

Ulrich bundles are always semistable  (see  \cite{C--H2}, Theorem 2.9), hence $\mu$--semistable. Corollary \ref{cLineBundle} shows that an actually even stronger property holds for Ulrich bundles on $F$

\begin{proposition}
\label{pStable}
Every Ulrich bundle $\cE$ of rank $2$ on $F$ is indecomposable and $\mu$--stable.
\end{proposition}
\begin{proof}
Let $\cE$ be an Ulrich bundle of rank $2$ on $F$. Assume that  $\cE$ is not  $\mu$--stable. It follows the existence of a torsion free quotient $\cM$ of $\cE$ of rank $1$ with $\mu(\cM)=\mu(\cE)$, hence $c_1(\cE)h^2$ is even. We have an exact sequence of the form
$$
0\longrightarrow\calL\longrightarrow \cE\longrightarrow\cM\longrightarrow0
$$
for a suitable sheaf $\calL$ of rank $1$ on $F$.

By the additivity of the first Chern class we obtain that $\mu(\calL)=\mu(\cE)$. Hence Theorem 2.8 of \cite{C--H2} implies that both $\calL$ and $\cM$ are both Ulrich bundles on $F$, contradicting Corollary \ref{cLineBundle}.

The same argument shows that if an Ulrich bundle of rank $2$ is decomposable, then its direct summands should necessarily be Ulrich line bundles. It thus follows from Corollary \ref{cLineBundle} that each Ulrich bundle of rank $2$ is necessarily indecomposable.
\end{proof}

\section{First properties of aCM bundles}
\label{sFirstProperties}
In this section we will inspect some properties of indecomposable, initialized, aCM bundles of rank $2$ on $F$. 

Let $\omega_2$ be the second Chern class of the sheaf $\Omega_{F}^1$. For each bundle $\mathcal H$ on $F$ with Chern classes $c_i\in A(F)$ Riemann--Roch theorem yields
\begin{equation}
  \label{RRgeneral}
    \chi(\mathcal H)=\rk(\mathcal H)+{1\over6}(c_1^3-3c_1c_2+3c_3)+{1\over2}(c_1^2h-2c_2h)+{1\over{12}}(4c_1h^2+\omega_2c_1).
\end{equation}

Let $\cE$ be a rank $2$, aCM bundle on $F$: throughout the whole paper we will denote by $c_1:=\alpha_1\xi +\alpha_2f$ and $c_2:=\beta_1\xi^2+\beta_2f^2$ its Chern classes. Easy computations yield
\begin{equation}
\label{chern}
\begin{gathered}
c_1^3=\alpha_1^3+3\alpha_1^2\alpha_2+3\alpha_1\alpha_2^2,\\
c_1^2h=2\alpha_1^2+4\alpha_1\alpha_2+\alpha_2^2,\\
c_1h^2=4\alpha_1+3\alpha_2.
\end{gathered}
\end{equation}
The canonical projection $\pi\colon F\cong\bP(\cO_{\p2}\oplus\cO_{\p2}(1))\to\p2$ is smooth, thus we also have the exact sequence
$$
0\longrightarrow \pi^*\Omega^1_{\bP^2} \longrightarrow \Omega^1_{F}\longrightarrow  \Omega^1_{F\vert\bP^2}\longrightarrow0.
$$
It follows that $\omega_2=6f\xi$, hence
$\omega_2c_1=6(\alpha_1+\alpha_2)$.

We have $h^i\big(F,\cE^\vee(-h)\big)=h^{3-i}\big(F,\cE(-h)\big)=0$ for $i\ge1$. Formula \eqref{RRgeneral} for $\cE^\vee(-h)$ and Equalities \eqref{chern} yield
\begin{equation}
\label{c_1c_2}
c_1c_2=2h^0\big(F,\cE^\vee(-h)\big)+\frac{1}{3}(\alpha_1^3+3\alpha_1^2\alpha_2+3\alpha_1\alpha_2^2-\alpha_1)=(\alpha_1+\alpha_2)\beta_1+\alpha_1\beta_2.
\end{equation}
Similarly $h^i\big(F,\cE^\vee\big)=h^{3-i}\big(F,\cE(-2h)\big)=0$ for $i\ge1$, whence 
\begin{equation}
\label{hc_2}
hc_2=h^0\big(F,\cE^\vee(-h)\big)-h^0\big(F,\cE^\vee\big)+\frac{1}{2}(2\alpha_1^2+4\alpha_1\alpha_2+\alpha_2^2-4\alpha_1-3\alpha_2+4)=2\beta_1+\beta_2.
\end{equation}

\begin{lemma}
  \label{lGG}
  Let $\cE$ be an initialized aCM bundle of rank $2$ on $F$ with $c_1(\cE)=\alpha_1 \xi+\alpha_2f$. Then $\cE^\vee(2h)$ is aCM, globally generated and $\alpha_i\le4$ for $i=1,2$.
\end{lemma}
\begin{proof}
The bundle $\cE^\vee$ is aCM by Serre duality. Moreover, $h^i\big(F,\cE^\vee((2-i)h)\big)=h^{3-i}\big(F,\cE((i-4)h)\big)=0$, $i=1,2,3$, thus $\cE$ is $2$--regular in the sense of Castelnuovo--Mumford (see \cite{Mu}), hence $\cE^\vee(2h)$ is globally generated.

It follows that $4h-c_1=c_1(\cE^\vee(2h))$ is globally generated. In particular $h^0\big(F,\cO_F(4h-c_1)\big)\ge1$, whence $\alpha_1\le4$ (see Corollary \ref{cNonVanishing}). 

We have $h^0\big(F,\cO_F(\xi-f)\big)=1$: the unique divisor $F_1\in \vert \xi-f\vert$ is a plane in $\p8$ because $(\xi-f)h^2=1$. Both line bundles $\cO_F(4h-c_1)$ and $\cO_F(f)$ are globally generated, thus for a general choice of $A\in \vert 4h-c_1\vert$ and $B\in \vert f\vert$ the intersection $F_1\cap A\cap B$ is proper. In particular $4-\alpha_2=(\xi-f)(4h-c_1)f\ge0$, whence $\alpha_2\le4$.
\end{proof}

We finally check that the zero--locus of each section of an indecomposable, initialized, aCM bundle of rank $2$ on $F$ is non--empty.

\begin{lemma}
\label{lNonEmpty}
Let $\cE$ be an indecomposable, initialized, aCM bundle of rank $2$ on $F$. Then the zero locus $(s)_0$ of a section of $\cE$ is non--empty.
\end{lemma}
\begin{proof}
Assume $(s)_0=\emptyset$. Sequence \eqref{seqIdeal} becomes
$$
0\longrightarrow \cO_F\longrightarrow \cE\longrightarrow \cO_F(c_1)\longrightarrow 0.
$$
Twisting it by $\cO_F(-c_1)$, taking the cohomology and the equality $h^1\big(F,\cE(-c_1)\big)=h^1\big(F,\cE^\vee\big)=h^2\big(F,\cE(-2h)\big)=0$, yield $h^1\big(F,\cO_F(-c_1)\big)\le h^0\big(F,\cO_F\big)=1$. The above sequence does not split because $\cE$ is indecomposable, thus equality holds. In particular each non--split extension of $ \cO_F(c_1)$ with $\cO_F$ is isomorphic to it.

Thanks to Corollary \ref{cNonVanishing}, we deduce that $\alpha_1\ge2$. The cohomology of the above sequence twisted by $\cO_F(-h)$ would yield $h^0\big(F,\cE(-h)\big)\ne0$, a contradiction because $\cE$ is assumed initialized. Thus $\alpha_1+\alpha_2\le 1$, again by Corollary \ref{cNonVanishing}. It follows that $\alpha_2=1-\alpha_1-u$ for a suitable non--negative $u$.

The equality $h^1\big(F,\cO_F(-c_1)\big)=1$ computed using Proposition \ref{pLineBundle} becomes
$$
\sum_{j=0}^{\alpha_1-2}{u+j+2\choose2}=1.
$$
From the above equality it is easy to deduce $u=0$ and $\alpha_1=2$, i.e. $c_1=2\xi-f$ necessarily.

Notice that $\cO_F(\xi)$ and $\cO_F(\xi-f)$ are effective and that $\cO_F(\xi)$ is globally generated. Since $\xi(\xi-f)=0$ and $\xi+(\xi-f)=c_1$, we deduce the existence of an exact sequence of the form
$$
0\longrightarrow \cO_F\longrightarrow \cO_F(\xi)\oplus\cO_F(\xi-f)\longrightarrow \cO_F(c_1)\longrightarrow 0.
$$
Such a sequence is non--split, thus $\cE\cong\cO_F(\xi)\oplus\cO_F(\xi-f)$, contradicting the indecomposability.
\end{proof}

As usual, let $\cE$ be a rank $2$ vector bundle on $F$, $c_1:=\alpha_1\xi +\alpha_2f$, $c_2:=\beta_1\xi^2+\beta_2f^2$ its Chern classes, $s\in H^0\big(F,\cE\big)$ and $(s)_0=E\cup D$
where $E$ has pure codimension $2$ (or it is empty) and $D$ has pure codimension $1$ (or it is empty). If $D\in\vert\delta_1\xi+\delta_2f\vert$, then \begin{equation}
\label{classE}
c_2(\cE(-D))=(\beta_1-\alpha_1\delta_2-\alpha_2\delta_1-\alpha_1\delta_1+2\delta_1\delta_2+\delta_1^2)\xi^2+(\beta_2-\alpha_2\delta_2+\delta_2^2)f^2.
\end{equation}

The bundles $\cO_F(f)$ and $\cO_F(\xi)$ are globally generated, thus Equality \eqref{classE} above implies
\begin{equation}
\label{positivity1}
\beta_1-\alpha_1\delta_2-\alpha_2\delta_1-\alpha_1\delta_1+2\delta_1\delta_2+\delta_1^2=c_2(\cE(-D))f\ge0,\\
\end{equation}
Moreover, if $E$ is not contained in the unique divisor $F_1\in\vert\xi-f\vert$, then also
\begin{equation}
\label{positivity2}
\beta_2-\alpha_2\delta_2+\delta_2^2=c_2(\cE(-D))(\xi-f)\ge0,
\end{equation}

\section{Effectiveness of $c_1-D$}
\label{sEffective}
For each non--zero $s\in H^0\big(F,\cE\big)$ we have $(s)_0=E\cup D$ where $E$ has pure codimension $2$ (or it is empty) and $D\in\vert\delta_1h_1+\delta_2h_2\vert$ has pure codimension $1$ (or it is empty): since $D$ is assumed to be effective, it follows that $\delta_1,\delta_1+\delta_2\ge0$. We will prove that $c_1-D$ is effective too.

\begin{lemma}
\label{dimension1}
Let $\cE$ be an indecomposable, initialized, aCM bundle of rank $2$ on $F$. Assume that $s\in H^0\big(F,\cE\big)$ is such that $(s)_0=E\cup D$ where $E$ has pure codimension $2$ and $D$ has pure codimension $1$. If $D\ne0$, then the class of $D$ is one of the following:
$$
f,\qquad 2f,\qquad 3f,\qquad 4f,\qquad \xi,\qquad \xi-f,\qquad 2\xi-f,\qquad 2\xi-2f.
$$
\end{lemma}
\begin{proof}
Twisting Sequence \eqref{seqIdeal} for $s$ by $\cO_F(2h-c_1)$ and taking into account that $\cE(-c_1)\cong{\cE}^\vee(2h)$ is globally generated (see Lemma \ref{lGG}), we  infer that $\cI_{E\vert F}(2h-D)$ is globally generated too. Thus
$$
0\ne H^0\big(F, \cI_{E\vert F}(2h-D)\big)\subseteq H^0\big(F, \cO_F(2h-D)\big),
$$
hence $\delta_1\le2$ and $\delta_1+\delta_2\le4$.

Twisting the same sequence by $\cO_F(-h)$ and
taking its cohomology, the vanishing of $h^0\big(F,\cE(-h)\big)$
implies $h^0\big(F,\cO_F(D-h)\big)=0$. In particular either $\delta_1=0$, or $\delta_1\ge1$ and $\delta_1+\delta_2\le1$. An easy computation gives now the statement.
\end{proof}

In order to study the effectiveness of $c_1-D$ it is helpful to look at the hyperplane sections of $F$ and $c_1-D$.

\begin{remark}
\label{rdelPezzo}
If $H\subseteq\p8$ is a general hyperplane, then $S:=F\cap H$ is a smooth del Pezzo surface of degree $7$. 

Recall that such an $S$ is the blow up of $\p2$ in $2$ distinct points embedded in $\p7$ via the linear system of cubics through such points. Thus its Picard group is freely generated by the class $\ell$ of the pull--back of a general line in $\p2$ and by the classes of the exceptional divisors $e_1$ and $e_2$, which are lines on $S$. We have $\ell^2=1$, $e_1^2=-1$ and $\ell e_i=0$, $i=1,2$. 

In particular each line bundle on $S$ is of the form $\cO_S(a\ell-b_1e_1-b_2e_2)$. Moreover, the class of the hyperplane section of $S$ is $3\ell-e_1-e_2$ and $\omega_S=\cO_S(-3\ell+e_1+e_2)$. 

It will be  important in what follows to know identify the classes of the divisors of $S$ inside $A(F)$. Clearly it suffice to know the classes of $\ell$ and $e_i$. Let $\epsilon_1\xi^2+\epsilon_2f^2$ be the class of the line $E$ in $A(F)$: we know that $1=Eh=2\epsilon_{1}+\epsilon_{2}$. Moreover both $\cO_F(f)$ and $\cO_F(\xi)$ are globally generated, thus $\epsilon_1=e_if\ge0$, $\epsilon_1+\epsilon_2=e_i\xi\ge0$. It is easy to conclude that the class of a line on $F$, in particular $e_i$, is either $f^2$, or $\xi^2-f^2$.

In the first case $E$ is the pull-back of a point $e\in\p2$ via $\pi$, thus it is the complete intersection inside $F$ of two divisors in $\vert f\vert$. In particular we have an exact sequence of the form
\begin{equation}
\label{f^2}
0\longrightarrow\cO_F(-2f)\longrightarrow\cO_F(-f)^{\oplus2}\longrightarrow\cI_{E\vert F}\longrightarrow0.
\end{equation}
In the second case $E(\xi-f)=-1$, thus $E$ is necessarily contained in the unique divisor $F_1\in\vert\xi-f\vert$. The unique section of $H^0\big(F,\cO_F(\xi-f)\big)$ induces a section of $\pi$ whose image is $F_1$. Thus $\pi_{\vert F_1}\colon F_1\to\p2$ is an isomorphism. It follows that $\pi_{\vert F_1}(E)\subseteq\p2$ is a line because $Ef=1$. We conclude that $E$ is the pull--back via $\pi_{\vert F_1}$ of a line in $\p2$. Hence $E$ is the complete intersection inside $F$ of $F_1$ with a divisor in $\vert f\vert$. In particular we have an exact sequence of the form
\begin{equation}
\label{xi^2-f^2}
0\longrightarrow\cO_F(-\xi)\longrightarrow\cO_F(-f)\oplus\cO_F(f-\xi)\longrightarrow\cI_{E\vert F}\longrightarrow0.
\end{equation}

The above computation proves the well--known fact that the Hilbert scheme $\mathcal H$ of lines in $F$ has two connected component $\mathcal H_1$ (the lines in the class $f^2$) and $\mathcal H_2$ (the lines in the class $\xi^2-f^2$) both isomorphic to $\p2$ (see \cite{I--P}, Proposition 3.5.6). 

Let $\eta_1\xi^2+\eta_2f^2$ be the class of $\ell$ inside $A(F)$: as above $\eta_1=\ell f\ge0$. Moreover, the elements in $\vert\ell\vert$ are rational and smooth, cubic curve on $S$, thus they are not contained in a plane: in particular they are not contained in the previously defined divisor $F_1$, thus we also know that $\eta_2=\ell(\xi-f)\ge0$. Finally, $3=\ell h=2\eta_1+\eta_2$, hence $\eta_2=3-2\eta_1$. 

Thus the class $3\ell-e_1-e_2$ of the hyperplane section in $A(F)$ is
$$
3(\eta_1\xi^2+(3-2\eta_1)f^2)-uf^2-(2-u)(\xi^2-f^2)
$$ 
for a suitable $u\in\{\ 0,1,2\ \}$. Since we also know that the class of $H$ in $A(F)$ is $h^2=3\xi^2+f^2$, we finally obtain that the classes of $\ell$ and $e_i$ are $\xi^2+f^2$ and $f^2$ respectively.
\end{remark}
\medbreak

We first look for non effective divisors on $F$ restricting to effective divisors on $S$.

\begin{lemma}
\label{ldelPezzo}
Let $\calL\cong\cO_F(\lambda_1\xi+\lambda_2f)$. Assume that
\begin{gather*}
h^0\big(F,\calL\big)=h^1\big(S,\cO_S\otimes\calL\big)=0,\qquad 
h^0\big(S,\cO_S\otimes\calL\big)\ge1.
\end{gather*}
Then $\lambda_1=-1$, $\lambda_2\ge2$ and
\begin{equation}
\label{hrestriction}
\begin{gathered}
h^0\big(S,\cO_S\otimes\calL\big)=h^1\big(F,\calL(-h)\big)-h^1\big(F,\calL\big),\\
h^1\big(S,\cO_S\otimes\calL(-h)\big)-h^0\big(S,\cO_S\otimes\calL(-h)\big)=h^1\big(F,\calL(-h)\big)-h^1\big(F,\calL(-2h)\big).
\end{gathered}
\end{equation}
\end{lemma}
\begin{proof}
Thanks to the hypothesis, the cohomology of sequence
\begin{equation}
\label{seqHyperplane}
0\longrightarrow \calL(-h)\longrightarrow \calL\longrightarrow \cO_S\otimes\calL\longrightarrow0
\end{equation}
yields the first of Equalities \eqref{hrestriction}.

Due to Corollary \ref{cNonVanishing} we have $\lambda_1\le -1$ and $\lambda_2\ge2$, because $h^1\big(F,\calL(-h)\big)\ge h^0\big(S,\cO_S\otimes\calL\big)\ge1$. As a consequence we also have $h^2\big(F,\calL(-2h)\big)=0$ (again by Corollary \ref{cNonVanishing}), thus the cohomology of Sequence \eqref{seqHyperplane} twisted by $\cO_F(-h)$ also gives the second of Equalities \eqref{hrestriction}.

We have that $h(\lambda_1\xi+\lambda_2f)=(2\lambda_1+\lambda_2)\xi^2+\lambda_2f^2$. As pointed out in Remark \ref{rdelPezzo} the classes of $\ell$ and $e_i$ in $A(F)$ are $\xi^2+f^2$ and $f^2$ respectively, thus 
$\cO_S\otimes\calL\cong\cO_S(C)$ where $C$ is a divisor in the linear system on $S$
$$
\vert(2\lambda_1+\lambda_2)\ell-ue_1-(2\lambda_1-u)e_2\vert
$$
for a suitable integer $u$. Such a divisor can be assumed effective, because $h^0\big(S,\cO_S\otimes\calL\big)\ge1$. Since $\omega_S=\cO_S(-3\ell+e_1+e_2)$, it follows that $\omega_S(-C)$ is not effective, hence
\begin{equation}
\label{vanishingh^2}
h^2\big(S,\cO_S\otimes\calL\big)=h^0\big(S,\omega_S(-C)\big)=0.
\end{equation}
By combining Proposition \ref{pLineBundle} with the first of Equalities \eqref{hrestriction} we obtain
\begin{align*}
h^0\big(S,\cO_S(C)\big)&=\sum_{j=0}^{-\lambda_1-1}{\lambda_1+\lambda_2+j+1\choose2}-\sum_{j=0}^{-\lambda_1-2}{\lambda_1+\lambda_2+j+3\choose2}=\\
&=\sum_{j=0}^{-\lambda_1-1}{\lambda_1+\lambda_2+j+1\choose2}-\sum_{j=2}^{-\lambda_1}{\lambda_1+\lambda_2+j+1\choose2}=\\
&={\lambda_1+\lambda_2+1\choose2}+{\lambda_1+\lambda_2+2\choose2}-{\lambda_2+1\choose2}.
\end{align*}
Combining such an equality with Riemann--Roch theorem on $S$, Equalities \eqref{vanishingh^2} and the vanishing $h^1\big(S,\cO_S\otimes\calL\big)=0$, we finally obtain the equality
\begin{align*}
2\lambda_1^2+4\lambda_1\lambda_2&+\lambda_2^2+4\lambda_1+3\lambda_2+2=2h^0\big(S,\cO_S(C)\big)=\\
&=2\chi(\cO_S(C))=4\lambda_1\lambda_2+\lambda_2^2+4\lambda_1+3\lambda_2+2+4\lambda_1u-2u^2,
\end{align*}
which trivially implies $u=\lambda_1$, which is at most $-1$, as checked above. Thus 
$$
h^0\big(S,\cO_S(C)\big)"=h^0\big(S,\cO_S((2\lambda_1+\lambda_2)\ell-\lambda_2(e_1+e_2))\big)\ge h^0\big(S,\cO_S((2\lambda_1+\lambda_2)\ell)\big).
$$
If $q\colon S\to \p2$ is the blow up map, then projection formula for $\cO_S((2\lambda_1+\lambda_2)\ell)\cong q^*(\cO_{\p2}(2\lambda_1+\lambda_2))$ yields
$$
2h^0\big(S,\cO_S(C)\big)\ge(2\lambda_1+\lambda_2+2)(2\lambda_1+\lambda_2+1).
$$
Combining such an inequality again with Riemann--Roch theorem on $S$ and the equality $u=\lambda_1$ we finally obtain that $\lambda_1(\lambda_1+1)\le0$. The unique integral solution of such an inequality in the range $\lambda_1\le-1$ is obviously $\lambda_1=-1$. This completes the proof of the statement.
\end{proof}

We are now ready to prove the main result of this section.

\begin{proposition}
\label{pEffective}
Let $\cE$ be an indecomposable, initialized, aCM bundle of rank $2$ on $F$. Assume that $s\in H^0\big(F,\cE\big)$ is such that $(s)_0=E\cup D$ where $E$ has codimension $2$ (or it is empty) and $D$ has codimension $1$ (or it is empty). If $c_1-D$ is not effective, then $E=\emptyset$.
\end{proposition}
\begin{proof}
Assume that $c_1-D\in\vert \lambda_1\xi+\lambda_2f\vert$ is not effective: thus
\begin{equation}
\label{NUM0}
h^0\big(F,\cI_{E\vert F}(c_1-D-th)\big)\le h^0\big(F,\cO_{F}(c_1-D-th)\big)=0
\end{equation}
for $t\ge0$. 

We set $A(D,t):=h^1\big(F,\cI_{E\vert F}(c_1-D-th)\big)$ and $B(D,t):=h^2\big(F,\cI_{E\vert F}(c_1-D-th)\big)$.
Taking into account that $\cE$ is aCM, the cohomology of Sequence \eqref{seqIdeal} gives the following bounds when $t=0,1,2$
\begin{equation}
\label{NUM1}
\begin{gathered}
A(D,t)=\left\lbrace\begin{array}{ll} 
1\quad&\text{if $D=2\xi-2f$ and $t=1$,}\\
0\quad&\text{if either $D\ne2\xi-2f$ and $t=1$, or $t=1$,}
  \end{array}\right.\\
B(D,t)\le\left\lbrace\begin{array}{ll} 
    1\quad&\text{if $D=0$ and $t=2$,}\\
0\quad&\text{if either $D\ne0$ and $t=2$, or $t=1$.}
  \end{array}\right.
\end{gathered}
\end{equation}

Let $H$ be a general hyperplane in $\p8$. Define $S:=F\cap H$ and $Z:=E\cap H$, so that $\dim(Z)=0$. The cohomology of 
$$
0\longrightarrow \cI_{E\vert F}(c_1-D-h)\longrightarrow \cI_{E\vert F}(c_1-D)\longrightarrow \cI_{Z\vert S}(c_1-D)\longrightarrow0
$$
and Relations \eqref{NUM0} and \eqref{NUM1} yield
\begin{equation}
\label{NUM3}
\begin{gathered}
h^0\big(S,\cI_{Z\vert S}(c_1-D)\big)\le A(D,1)\le1,\\
h^1\big(S,\cI_{Z\vert S}(c_1-D)\big)\le A(D,0)+B(D,1)=0,\\
h^1\big(S,\cI_{Z\vert S}(c_1-D-h)\big)\le A(D,1)+B(D,2)\le1.
\end{gathered}
\end{equation}

Now consider the exact sequence 
$$
0\longrightarrow \cI_{Z\vert S}(c_1-D)\longrightarrow \cO_S(c_1-D)\longrightarrow \cO_Z\longrightarrow0
$$
The cohomology of the above sequence, the equality $\dim(Z)=0$, and the second Relation \eqref{NUM3} implies that 
\begin{equation}
\label{NUM6}
h^1\big(S,\cO_S(c_1-D)\big)=0.
\end{equation}
All such relations thus give
\begin{equation}
\label{NUM5}
\begin{gathered}
\deg(Z)=h^0\big(Z,\cO_Z\big)=h^0\big(S,\cO_S(c_1-D)\big)-h^0\big(S,\cI_{Z\vert S}(c_1-D)\big),\\
{\begin{aligned}
\deg(Z)=h^0\big(Z,\cO_Z(-h)\big)&=h^0\big(S,\cO_S(c_1-D-h)\big)-h^0\big(S,\cI_{Z\vert S}(c_1-D-h)\big)-\\
&-h^1\big(S,\cO_S(c_1-D-h)\big)+h^1\big(S,\cI_{Z\vert S}(c_1-D-h)\big).
\end{aligned}}
\end{gathered}
\end{equation}

Let us now assume that $E\ne\emptyset$ so that 
$$
h^0\big(S,\cO_S(c_1-D)\big)\ge\deg(Z)=\deg(E)\ge1.
$$
Since $c_1-D$ is not effective, it follows from the above inequality and from Equality \eqref{NUM6} that Lemma \ref{ldelPezzo} holds in this case. 

In particular $\lambda_1=-1$ and $\lambda_2\ge2$, thus
\begin{align*}
\Delta:&=2h^1\big(F,\cO_F(c_1-D-h)\big)-h^1\big(F,\cO_F(c_1-D)\big)-h^1\big(F,\cO_F(c_1-D-2h)\big)=\\
&=2{\lambda_2\choose2}-{\lambda_2-1\choose2}-{\lambda_2-2\choose2}.
\end{align*}
Equating the last members of Identities \eqref{NUM5} and taking into account Equalities \eqref{hrestriction} we obtain
$$
\Delta=h^0\big(S,\cI_{Z\vert S}(c_1-D)\big)-h^0\big(S,\cI_{Z\vert S}(c_1-D-h)\big)+h^1\big(S,\cI_{Z\vert S}(c_1-D-h).
$$
Taking into account of Inequalities \eqref{NUM3} we deduce the bound $\Delta\le 2A(D,1)+B(D,2)\le2$,
which forces $\lambda_2\le2$. Since we already know that $\lambda_2\ge2$, it follows that equality holds. Thus $\Delta=2$, whence the above bound yields $A(D,1)=1$, i.e $D=2\xi-2f$ and $c_1=\xi$.

We have $h^0\big(F,\cO_F(-2\xi+2f)\big)=0$ by Corollary \ref{cNonVanishing}, thus the cohomology of Sequence \eqref{seqIdeal} twisted by $\cO_F(-c_1)$ finally yields 
$$
h^0\big(F,\cE^\vee(-h)\big)\le h^0\big(F,\cE^\vee\big)=0.
$$
Equalities \eqref{c_1c_2} and \eqref{hc_2} returns $c_2=\xi^2-f^2$. In particular the class of $E$ would be $3f^2-3\xi^2$, contradicting Inequality \eqref{positivity1}. The contradiction implies that $E=\emptyset$ necessarily.
\end{proof}

It follows from the above Proposition that Sequence \eqref{seqIdeal} becomes
\begin{equation}
\label{seqIdealO}
0\longrightarrow \cO_F(D)\longrightarrow \cE\longrightarrow \cO_{F}(c_1-D) \longrightarrow 0.
\end{equation}
We know that $h^1\big(F,\cO_F(2D-c_1)\big)\ne0$, because $\cE$ is supposed to be indecomposable, hence $\delta_1-\lambda_1\le -2$ (see Corollary \ref{cNonVanishing}).
It follows that
\begin{equation}
\label{DIS1}
\lambda_1\ge2,
\end{equation}
because we know that $\delta_1\ge0$. The vanishing $h^0\big(F,\cO_F(c_1-D)\big)=0$ and Corollary \ref{cNonVanishing} thus imply
\begin{equation}
\label{DIS2}
\lambda_1+\lambda_2\le-1.
\end{equation}

Inequalities \eqref{DIS1}, \eqref{DIS2}, Corollary \ref{cNonVanishing} and Lemma \eqref{dimension1} give $h^2\big(F,\cO_F(D)\big)=h^3\big(F,\cO_F(D)\big)=0$. Moreover, the cohomology of Sequence \eqref{seqIdealO} gives $h^2\big(F,\cO_F(c_1-D)\big)=0$, being $\cE$ aCM. 

Thus again Corollary \ref{cNonVanishing} implies
\begin{equation}
\label{DIS3}
\lambda_2\ge-2.
\end{equation}
Since the three Inequalities \eqref{DIS1}, \eqref{DIS2}, \eqref{DIS3} have no common solutions, we have finally proved the following result.

\begin{corollary}
\label{cEffective}
Let $\cE$ be an indecomposable, initialized, aCM bundle of rank $2$ on $F$. Assume that $s\in H^0\big(F,\cE\big)$ is such that $(s)_0=E\cup D$ where $E$ has codimension $2$ (or it is empty) and $D$ has codimension $1$ (or it is empty). Then $c_1-D$ is effective.
\end{corollary}

\section{Positivity of $c_1$}
\label{sTheoremAB}
As usual if $\cE$ is a vector bundle of rank $2$ on $F$, we will assume that its Chern classes are $c_1:=\alpha_1\xi +\alpha_2f$ and $c_2:=\beta_1\xi^2+\beta_2f^2$. In this section we will prove that $c_1$ is actually globally generated and bounded. Moreover we will deal with the zero locus of its general sections.

\begin{theorem}
\label{tBound}
If $\cE$ is an indecomposable, initialized, aCM bundle of rank $2$ on $F$, then $0\le \alpha_1,\alpha_2\le2$. 

Moreover, if $\alpha_1=\alpha_2=2$, then $\cE$ is Ulrich.
\end{theorem}
\begin{proof}
Let $s\in H^0\big(F,\cE\big)$ be general: as usual we write $(s)_0=E\cup D$. We already know that $c_1-D$ is effective (see Corollary \ref{cEffective}): if $\lambda_i=\alpha_i-\delta_i$, then we know from Corollary \ref{cNonVanishing} that $\lambda_1\ge0$ and $\lambda_1+\lambda_2\ge0$. It follows that
$$
\alpha_1\ge \delta_1\ge0,\qquad \alpha_2\ge\delta_1+\delta_2-\alpha_1\ge-\alpha_1.
$$

We begin by proving that $\alpha_2\ge0$ too. This follows trivially from the above inequalities, if $\alpha_1=0$, thus we will assume $\alpha_1\ge1$ in what follows.

Recall that $\alpha_1\le4$ (see Lemma \ref{lGG}). Moreover, the cohomology of Sequence \eqref{seqIdeal} twisted by $\cO_F(-c_1)$ and the effectiveness of $c_1-D$ give
\begin{equation}
\label{D=c_1}
\begin{gathered}
h^0\big(F,\cE^\vee(-h)\big)=h^0\big(F,\cE^\vee\big)=0\qquad\text{if $D\ne c_1$,}\\
h^0\big(F,\cE^\vee(-h)\big)<h^0\big(F,\cE^\vee\big)=1\qquad\text{if $D=c_1$.}
\end{gathered}
\end{equation}
Examining the possible cases by taking into account Lemma \ref{dimension1}, the above bounds, the previous vanishing and Equalities \eqref{c_1c_2} and \eqref{hc_2}, we obtain the following table.
\vskip3truemm
\centerline{Table A}
\vglue-3truemm
$$
\begin{array}{|c|c|c|c|c|c|}           \hline
{ (\alpha_1,\alpha_2) }		& D=c_1				& { (\beta_1,\beta_2) }	\\ \hline
(1,-1)					& \mathrm{no}			& (1/2,0) 				\\  \hline
(1,-1)					& \mathrm{yes}			& (0,0) 				\\  \hline
(2,-1)					& \mathrm{no}			& (0,0) 				\\  \hline
(2,-1)					& \mathrm{yes}			& (-2/3,1/3) 			\\  \hline
(2,-2)					& \mathrm{no}			& (-1,1) 				\\  \hline
(2,-2)					& \mathrm{yes}			& (-3/2,1) 				\\  \hline
(3,-1)					& \mathrm{no}			& (1/4,1/2) 			\\  \hline
(3,-2)					& \mathrm{no}			& (-8/5,6/5) 			\\  \hline
(3,-3)					& \mathrm{no}			& (-10/3,8/3) 			\\  \hline
(4,-1)					& \mathrm{no}			& (8/5,4/5) 			\\  \hline
(4,-2)					& \mathrm{no}			& (-4/3,5/3) 			\\  \hline
(4,-3)					& \mathrm{no}			& (-4,3) 				\\  \hline
(4,-4)					& \mathrm{no}			& (-13/2,5) 			\\  \hline
\end{array}
$$
\vskip3truemm
\noindent We conclude that only the cases $(1,-1)$ and $D=c_1$, $(2,-1)$ and $D\ne c_1$, $(2,-2)$, $(4,-3)$ must be examined.

In the first case $D=c_1=\xi-f$. The class of $E$ in $A(F)$ is $0$ (see Equality \eqref{classE}), thus Sequence \eqref{seqIdeal} coincides with Sequence \eqref{seqIdealO}, i.e.
$$
0\longrightarrow\cO_F(\xi-f)\longrightarrow\cE\longrightarrow\cO_F\longrightarrow0.
$$
Since $h^1\big(F,\cO_F(2D-c_1)\big)=h^1\big(F,\cO_F(\xi-f)\big)=0$ (see Corollary \ref{cNonVanishing}), it follows that $\cE$ split in this case. 

In the second case $D\ne c_1=2\xi-f$: thus $D$ is either $0$, or its class is $f$, or $\xi$ or $\xi-f$, because $c_1-D$ is effective. If $D=0$, then the class of $E$ in $A(F)$ is $0$ (see Equality \eqref{classE} again), thus $E=\emptyset$, hence $(s)_0=\emptyset$. This contradict Lemma \ref{lNonEmpty}. If $D=f$, then the class of $E$ would be $c_2(\cE(-D))=-2\xi^2+2f^2$, contradicting Inequality \eqref{positivity1}. In the two last cases it is easy to check that the class of $E$ in $A(F)$ is $c_2(\cE(-D))=0$. If $\cO_F(D)=\cO_F(\xi)$, then  Sequence \eqref{seqIdeal} becomes
$$
0\longrightarrow\cO_F(\xi)\longrightarrow\cE\longrightarrow\cO_F(\xi-f)\longrightarrow0.
$$
We exclude this case, because $h^1\big(F,\cO_F(2D-c_1)\big)=h^1\big(F,\cO_F(f)\big)=0$, hence the above sequence splits. The case $\cO_F(D)=\cO_F(\xi-f)$ can be excluded with a similar argument.

In the third case $c_1=2\xi-2f$ and $D\ne c_1$, thus $D$ is either $0$, or its class is $\xi-f$. Thanks to Inequality \eqref{positivity1} we can assume $\cO_F(D)=\cO_F(\xi-f)$. If this is the case, then the class of $E$ in $A(F)$ is $c_2(\cE(-D))=0$ (see Equality \eqref{classE}). Sequence \eqref{seqIdealO} is
$$
0\longrightarrow\cO_F(\xi-f)\longrightarrow\cE\longrightarrow\cO_F(\xi-f)\longrightarrow0
$$
which necessarily splits because $h^1\big(F,\cO_F(2D-c_1)\big)=h^1\big(F,\cO_F\big)=0$. 

In the last case $c_1=4\xi-3f$ and $D\ne c_1$: thus $D$ is either $0$, or its class is $f$, or $\xi$, or $\xi-f$, or $2\xi-f$, or $2\xi-2f$. Again using Inequality \eqref{positivity1} as above one checks that all these cases cannot occur.

Up to now we have thus proved that $\alpha_i\ge0$, $i=1,2$. In what follows we will prove that $\alpha_i\le2$, $i=1,2$. To this purpose we distinguish two cases according to whether $\cE$ is regular in the sense of Castelnuovo--Mumford or not. 

In the latter case $h^0\big(F,\cE^\vee(h)\big)=h^3\big(F,\cE(-3h)\big)\ne0$ necessarily, because $\cE$ is aCM. Thus, if $t\in\bZ$ is such that $\cE^\vee(th)$ is initialized, we know that $t\le1$. Notice that  $\cE^\vee(th)$ is aCM too and that $c_1(\cE^\vee(th))=(2t-\alpha_1)\xi+(2t-\alpha_2)f$. We conclude from the first part of this proof and from the inequality $t\le1$ that $2-\alpha_i\ge0$, $i=1,2$, i.e. the desired upper bound.

Let $\cE$ be regular, whence globally generated. Thus the zero--locus $E:=(s)_0$ of a general section $s\in H^0\big(F,\cE\big)$ is a smooth curve: in particular its divisorial part $D$ is zero. If $\alpha_1=\alpha_2=0$ there is nothing to prove: thus we will assume in what follows that at least one of the $\alpha_i$'s is positive. The cohomology of Sequence \eqref{seqIdeal} twisted by $\cO_F(-c_1)$ yields $h^0\big(F,\cE^\vee\big)=h^0\big(F,\cO(-c_1)\big)=0$.

Combining the above vanishing with equality \eqref{RRgeneral}, Formulas  \eqref{c_1c_2} and \eqref{hc_2}, we finally obtain
$$
14-4\alpha_1-3\alpha_2=\chi(\cE^\vee(h))=0,
$$
because $h^i\big(F,\cE^\vee(h)\big)=h^{3-i}\big(F,\cE(-3h)\big)=0$, $i =0, 1, 2, 3$: indeed $\cE$ is aCM, regular and initialized.

We know that $\cE$ is globally generated, thus the same is true for $\cO_F(c_1)$. Moreover, $\cO_F(f)$ is globally generated too. For a general choice of $A\in \vert\alpha_1\xi+\alpha_2 f\vert$ and $B\in\vert f\vert$ the intersection $F_1\cap A\cap B$ is proper, thus  $\alpha_2=(\xi-f)c_1f\ge0$. Taking into account of this restriction, the equality above yields $\alpha_1=\alpha_2=2$, i.e. $c_1=2h$.

We conclude the proof by examining the case of aCM bundles $\cE$ is initialized, aCM with $c_1=2h$. We have $\cE^\vee\cong\cE(-2h)$, hence $h^3\big(F,\cE(-3h)\big)=h^0\big(F,\cE^\vee(h)\big)=h^0\big(F,\cE(-h)\big)=0$. It follows that $\cE$ is regular, whence globally generated. We can thus again conclude that the zero locus of a general section of $\cE$ is a smooth curve $E$: in particular the divisorial part $D$ is zero so that Equalities \eqref{hc_2} and \eqref{c_1c_2} imply that $hc_2=9$, $c_1c_2=18$. A direct substitution in Formula \eqref{RRgeneral} finally implies that $h^0\big(F,\cE\big)=\chi(\cE)=14$, hence $\cE$ is Ulrich. 
\end{proof}

Let $\cE$ be an indecomposable, initialized, aCM bundle of rank $2$ on $F$. As usual, $c_1=\alpha_1\xi+\alpha_2f$ and $c_2=\beta_1\xi^2+\beta_2f^2$ are the Chern classes of $\cE$. Moreover, if $s\in H^0\big(F,\cE\big)$ is a general section of $\cE$ we write as above $(s)_0=E\cup D$. 

If $c_1=0$, Proposition \ref{pEffective} implies that $D=0$. If $c_1=2h$, then Lemma \ref{lGG} implies that $\cE\cong\cE^\vee(2h)$ is globally generated, thus again $D=0$ thanks to Bertini's theorem for the sections of a vector bundle.
In the remaining cases Lemma \ref{dimension1}, Relations \eqref{D=c_1} and Equalities \eqref{c_1c_2}, \eqref{hc_2} allow us to list all the possible cases in the following table.
\vskip3truemm
\centerline{Table B}
\vglue-3truemm
$$
\begin{array}{|c|c|c|c|c|c|}           \hline
\mathrm{Name} 			&{ (\alpha_1,\alpha_2) }		& D=c_1				& { (\beta_1,\beta_2) }	\\ \hline
\mathrm{A}_0				& (0,1)					& \mathrm{no}			& (0,1) 				\\  \hline
\mathrm{A}_1				& (1,0)					& \mathrm{no}			& (1,-1) 				\\  \hline
\mathrm{A}_2				& (0,2)					& \mathrm{no}			& (0,1) 				\\  \hline
\mathrm{A}_3				& (1,1)					& \mathrm{no}			& (\beta,2(1-\beta)) 		\\  \hline
\mathrm{A}_4				& (2,0)					& \mathrm{no}			& (1,0) 				\\  \hline
\mathrm{A}_5				& (1,2)					& \mathrm{no}			& (2,0) 				\\  \hline
\mathrm{A}_6				& (2,1)					& \mathrm{no}			& (2,1) 				\\  \hline
\mathrm{A}_7				& (0,1)					& \mathrm{yes}			& (0,0) 				\\  \hline
\mathrm{A}_8				& (1,0)					& \mathrm{yes}			& (0,0) 				\\  \hline
\mathrm{A}_9				& (0,2)					& \mathrm{yes}			& (0,0) 				\\  \hline
\end{array}
$$
\vskip3truemm

\begin{remark}
\label{rDual}
Notice that  $\cE$ is an indecomposable, aCM bundle of rank $2$ on $F$ if and only if the same holds for 
$\cE^\vee(th)$ and $t\in\bZ$. Moreover $\cE^\vee(th)$ is also initialized for exactly one $t\in \bZ$. The Chern classes of $\cE^\vee(th)$ are
\begin{equation}
\label{dualChern}
\begin{gathered}
c_1(\cE^\vee(th))=(2t-\alpha_1)\xi+(2t-\alpha_2)f,\\
c_2(\cE^\vee(th))=(\beta_1-t(2\alpha_1+\alpha_2)+3t^2)\xi^2+(\beta_2-t\alpha_2+t^2)f^2,\\
\end{gathered}
\end{equation}

The above table and Equalities \eqref{dualChern} then show that only the cases up to $\mathrm{A}_6$ can actually occur. Moreover, case $\mathrm{A}_i$ occurs if and only if case $\mathrm{A}_{6-i}$ occurs too, for $i=0,\dots,6$.
\end{remark}

\begin{theorem}
\label{tDzero}
If $\cE$ is an indecomposable, initialized, aCM bundle of rank $2$ on $F$, then the zero locus of a general section $s\in H^0\big(F,\cE\big)$ has pure codimension $2$.
\end{theorem}
\begin{proof}
Let $(s)_0=E\cup D$, as usual.

We already proved the statement if $c_1$ is either $0$ or $D$. In the case $\mathrm{A}_0$ there is nothing to prove, because $D\ne c_1$ and $c_1-D$ is effective. Thus we can restrict to the other cases listed above, assuming $D\ne0$.

In the case {A}$_1$, the class of $D$ is $f$, due to Lemma \ref{dimension1}. Equality \eqref{classE} implies that $E=\emptyset$. Sequence \eqref{seqIdeal} coincides with Sequence \eqref{seqIdealO}. The equality $h^1\big(F,\cO_F(2D-\xi)\big)=h^1\big(F,\cO_F(2f-\xi)\big)=0$ (see Corollary \ref{cNonVanishing}), implies that $\cE$ splits in this case. An analogous argument holds also in the case {A}$_2$, because only the case $D=f$ must be considered.

Let us examine the case  {A}$_3$. Due to Lemma \ref{dimension1}, then the class of the divisorial part $D$ can be one of the following: $f$, $2f$, $\xi$, $\xi-f$. In all these cases the class of $E$ is either $(\beta-1)(\xi^2-2f^2)$, or $(\beta-2)(\xi^2-2f^2)$. The line bundle $\cO_F(\xi)$ is globally generated, thus $c_2(\cE(-D))\xi\ge0$. Such an inequality and Inequality \eqref{positivity1} that $\beta$ is either $1$, or $2$ respectively. It follows that $E=\emptyset$, hence $\cE$ fits into Sequence \eqref{seqIdealO}. The vanishing $h^1\big(F,\cO_F(2D-c_1)\big)=0$ used above proves that $\cE$ splits also in these cases.

Let us examine the case {A}$_4$. The class of $D$ can be one of the following: $f$, $2f$, $2\xi-f$, $2\xi-2f$, $\xi$, $\xi-f$ again by due to Lemma \ref{dimension1}. Inequalities \eqref{positivity1} imply that the first four cases cannot occur. In the fifth case we have $c_2(\cE(-\xi))=0$, and we can still use the argument above because $h^1\big(F,\cO_F(2D-2\xi)\big)=h^1\big(F,\cO_F\big)=0$.  In the last case, we have $c_1(\cE(-D))=2f$ and $c_2(\cE(-D))=f^2$. Besides  the exact sequence
$$
0\longrightarrow\cO_F\longrightarrow\cE(-D)\longrightarrow\cI_{E\vert F}(2f)\longrightarrow0,
$$
we also have Sequence \eqref{f^2}. Thanks to Theorem \ref{tSerre}, it follows that $\cE=\cO_F(\xi)^{\oplus2}$, because $h^1\big(F,\cO_F(-2f)\big)=0$ (see Corollary \ref{cNonVanishing}).

In the case {A}$_5$ the bundle $\cE':=\cE^\vee(h)\cong \cE(-f)$ is in the case A$_1$. The zero locus $L$ of a general section of $\cE'$ is a line in the class $\xi^2-f^2$. Twisting Sequence \eqref{xi^2-f^2} by $\cO_F(h)$ we deduce that $\cI_{L\vert F}(h)$ is globally generated. On the other hand $\cO_F(f)$ is globally generated and $h^1\big(F,\cO_F(f)\big)=0$. Taking Sequence \eqref{seqIdeal} for $\cE'$ and twisting it by $\cO_F(f)$, it follows the existence of a commutative diagram of the form
$$
\begin{CD}
 0@>>>  H^0\big(F,\cO_F(f)\big)\otimes\cO_F@>>> H^0\big(F,\cE\big)\otimes\cO_F@>>> H^0\big(F,\cI_{L\vert F}(h)\big)\otimes\cO_F @>>>  0\\
 @. @VVV @VVV@VVV\\
0@>>>  \cO_F(f)@>>> \cE@>>> \cI_{L\vert F}(h) @>>>  0\\
\end{CD}
$$
Since the first and third vertical arrows are surjective, it follows that the same is true for the middle one. We conclude that $\cE$ is globally generated, thus the zero locus of the  general section $s\in H^0\big(F,\cE\big)$ has pure codimension $2$.

In the case A$_6$ we can argue similarly using Sequence \eqref{f^2} instead of \eqref{xi^2-f^2}. Indeed in this case $\cE':=\cE^\vee(h)\cong \cE(-\xi)$ and the zero locus $L$ of a general section of $\cE'$ is a line in the class $f^2$.
\end{proof}

\section{Elliptic curves and Ulrich bundles}
\label{sUlrich}
In this section we will deal with the zero locus of a general section of an Ulrich bundle of rank $2$ on $F$. We recall that Ulrich bundles of rank $2$ on $F$ are necessarily indecomposable (see Proposition \ref{pStable}).

Let us consider the case $c_1=2h$, so that $\cE$ is Ulrich by Theorem \ref{tBound}. In this case Lemma \ref{lGG} guarantees that $\cE\cong\cE^\vee(2h)$ is globally generated, thus the zero locus of a general section $s\in H^0\big(F,\cE\big)$ vanishes exactly along a smooth curve $E$. The cohomology of Sequence \eqref{seqIdeal} imply that $h^1\big(F,\cI_{E\vert F}\big)=h^0\big(F,\cI_{E\vert F}(h)\big)=0$. Hence the cohomology of Sequence \eqref{seqStandard} yields both $h^0\big(E,\cO_{E}\big)=1$ and that $E$ is non--degenerate (i.e. not contained in any hyperplane). Moreover we now that $\deg(E)=hc_2=9$ (see Equality \eqref{hc_2}). Isomorphism \eqref{NormalBundle} finally yields $\omega_E=\cO_E$. We conclude that $E$ is a non--degenerate elliptic curve in $\p8$.

Conversely each such a curve  $E$ on $F$ arises in the above way. Indeed $\cO_E\cong\omega_E\cong\det(\cN_{E\vert F})\otimes\cO_F(-2h)$, hence $\det(\cN_{E\vert F})\cong\cO_F(2h)\otimes\cO_E$. Moreover, $\cO_F(2h)$ satisfies the vanishing of the Theorem \ref{tSerre}, thus $E$ is the zero locus of a section $s$ of a vector bundle $\cE$ of rank $2$ on $F$ with $c_1=2h$. Such a bundle is initialized because we are assuming $E$ linearly normal.
Proposition 1.1 and Corollary 2.2 of \cite{C--G--N} imply that $E$ is aCM, then $h^1\big(\p8,\cI_{E\vert\p8}(th)\big)=0$. The cohomology of sequence
\begin{equation*}
\label{seqIdealEF}
0\longrightarrow \cI_{F\vert\p8}\longrightarrow \cI_{E\vert\p8}\longrightarrow \cI_{E\vert F}\longrightarrow 0,
\end{equation*}
imply $h^1\big(F,\cI_{E\vert F}(th)\big)=0$, because $h^2\big(\p8,\cI_{F\vert\p8}(t)\big)=0$. The cohomology of the Sequence \eqref{seqIdeal} yields $h^1\big(F,\cE(th)\big)=0$, hence  $h^2\big(F,\cE(th)\big)=0$ by Serre's duality. Thus $\cE$ is aCM with $c_1=2h$, whence Ulrich (again by Theorem \ref{tBound}). 

In particular, non--degenerate elliptic curves (hence normal) $E\subseteq\p8$ correspond to Ulrich bundles of rank $2$  on $F$. 
\begin{lemma}
\label{lUlrichc_2}
If $\cE$ is an Ulrich bundle of rank $2$ with $c_1=2h$, then $c_2$ is either $3\xi^2+3f^2$, or $4\xi^2+f^2$.
\end{lemma}
\begin{proof}
Let $c_2=\beta_1\xi^2+\beta_2 f^2$. Since $E$ is non--degenerate, then $E\not\subseteq F_1$, thus $\beta_2\ge0$, due to Inequality \eqref{positivity2}. Each divisor in $\vert f\vert$ has degree $3$, thus it is necessarily contained in a hyperplane of $\p8$. As above we conclude that $E$ cannot be contained in any divisor in $\vert f\vert$, thus we also have $\beta_1\ge0$, due to Inequality \eqref{positivity1}. Moreover, the restriction of $\vert f\vert$ to $E$ is a linear series of dimension $2=\dim\vert f\vert$ hence its degree $\beta_1=c_2f$ is at least $3$, because $E$ is elliptic. We obtain that $(\beta_1,\beta_2)$ must be either $(3,3)$, or $(4,1)$, because Equality \eqref{hc_2} gives $2\beta_1+\beta_2=9$.
\end{proof}

In what follows we will prove the existence of Ulrich bundles on $F$ with $c_2$ either $3\xi^2+3f^2$, or $4\xi^2+f^2$. Thanks to the above lemma, this is equivalent to prove the existence of elliptic normal curves on $F$ whose classes in $A^2(F)$ are either $3\xi^2+3f^2$, or $4\xi^2+f^2$.

\subsection{The case $3\xi^2+3f^2$}
The following results have been strongly inspired by the method used in \cite{C--MR--PL} and \cite{MR--PL2} for constructing families of Ulrich bundles on $\p2\times\p m$ and on Fano blow--ups of $\p n$.

Let us consider a stable vector bundle $\cF$ of rank $2$ on $\p2$ with $c_1(\cF)=0$ and $c_2(\cF)=2$. In particular we know that $h^0\big(\p2,\cF\big)=0$. The zero locus $Z$ of a general section of $\cF(1)$ is a $0$--dimensional scheme in $\p2$ of degree $3$ and we have an exact sequence of the form
\begin{equation}
\label{seqBundle2}
0\longrightarrow\cO_{\p2}\longrightarrow\cF(1)\longrightarrow\cI_{Z\vert\p2}(2)\longrightarrow0,
\end{equation}
If $Z$ is contained in a line, then $h^0\big(\p2,\cF\big)\ne0$, a contradiction. We have an exact sequence
$$
0\longrightarrow\cI_{Z\vert\p2}\longrightarrow\cO_{\p2}\longrightarrow \cO_Z\longrightarrow0.
$$
We conclude that
\begin{gather*}
h^1\big(\p2,\cF\big)=h^1\big(\p2,\cI_{Z\vert\p2}(1)\big)=h^0\big(\p2,\cI_{Z\vert\p2}(1)\big)=0,\\
h^2\big(\p2,\cF(-1)\big)=h^2\big(\p2,\cI_{Z\vert\p2}\big)=0,
\end{gather*}
thus $\cF(1)$ is regular in the sense of Castelnuovo--Mumford, hence it is globally generated. It follows that $Z$ can be assumed reduced, hence $h^0\big(\p2,\cI_{Z\vert\p2}(3)\big)=3$. The cohomology of Sequence \eqref{seqBundle2} thus yields  a surjection $\varphi\colon S^{\oplus4}\to H^0_*\big(\p2,\cF(1)\big)$, where $S:=k[x_0,x_1,x_2]$, which is an isomorphism in degree $0$. After sheafifying such a surjection we obtain an exact sequence
$$
0\longrightarrow\mathcal K\longrightarrow\cO_{\p2}^{\oplus4}\longrightarrow \cF(1)\longrightarrow0
$$
where $\mathcal K$ is a vector bundle of rank $2$. The properties of the map $\varphi$ and Horrocks theorem on $\p2$ thus imply that $\mathcal K$ splits as $\cO_{\p2}(-1)^{\oplus2}$. We  thus have an  exact sequence
$$
0\longrightarrow\cO_{\p2}(-2)^{\oplus2}\longrightarrow\cO_{\p2}(-1)^{\oplus4}\longrightarrow\cF\longrightarrow0.
$$
Let $\cE:=(\pi^*\cF)(h)$. By pulling back via $\pi$ the above sequence and twisting the result by $\cO_F(h)$, we obtain the exact sequence
\begin{equation}
\label{seq33}
0\longrightarrow\cO_F(\xi-f)^{\oplus2}\longrightarrow\cO_F(\xi)^{\oplus4}\longrightarrow\cE\longrightarrow0.
\end{equation}

\begin{proposition}
\label{p33}
There exists Ulrich bundles $\cE$ of rank $2$ on $F$ with $c_1=2h$ and $c_2=3\xi^2+3f^2$. 

Moreover each such a bundle arises via the construction described above.
\end{proposition}
\begin{proof}
The assertions on the rank and the Chern classes of $\cE$ follow immediately from the construction above. 

The line bundles $\cO_F(\xi)$ and $\cO_F(\xi-f)$ are aCM (see Corollary \ref{cLineBundle}), thus $h^1\big(F,\cO_F(th+\xi)\big)=h^2\big(F,\cO_F(th+\xi-f)\big)=0$, $t\in \bZ$. It follows that the cohomology of Sequence \eqref{seq33} gives $h^1\big(F,\cE(th)\big)=0$, $t\in \bZ$. Since $\cE\cong\cE^\vee(2h)$, then $h^2\big(F,\cE(th)\big)=h^1\big(F,\cE^\vee((-2-t)h)\big)=h^1\big(F,\cE^\vee((-4-t)h)\big)=0$ thanks to the previous vanishing. We deduce that $\cE$ is aCM. 

By computing the cohomology of Sequence \eqref{seq33}, one also obtains $h^0\big(F,\cE\big)=14$ and $h^0\big(F,\cE(-h)\big)=0$. We deduce that $\cE$ is Ulrich. Thus the existence part of the statement is completely proved.

Conversely, let us consider an Ulrich bundle $\cE$ of rank $2$ on $F$ with $c_1=2h$ and $c_2=3\xi^2+3f^2$ and let $E$ be the zero--locus of a general section of $E$: we know that $E$ is an elliptic normal curve in $\p8$, hence it is non--degenerate. Recall that we have a natural section $q\colon \p2\to F$ of $\pi$ mapping $\p2$ isomorphically onto the unique divisor $F_1\in\vert\xi-f\vert$ which we will identify with $\p2$ (see the proof of Lemma \ref{lGG}). If we set $\cF:=q^*(\cE(-h))$, then $c_1(\cF)=0$ and $c_2(\cF)=2$. Moreover, we have the exact sequence
\begin{equation}
\label{seqRestr}
0\longrightarrow\cO_F(f-\xi)\longrightarrow\cO_F\longrightarrow\cO_{\p2}\longrightarrow0.
\end{equation}
Tensorizing the above sequence by $\cE(-h)$, taking its cohomology and the one of Sequence \eqref{seqIdeal} twisted by $\cO_F(-2\xi)$, we obtain $h^0\big(\p2,\cF\big)=h^1\big(F,\cE(-2\xi)\big)=h^0\big(F,\cI_{E\vert F}(2f)\big)$. Notice that each surface in $\vert 2f\vert$ has degree $6$, thus it is degenerate in $\p8$. Hence the last dimension must be zero because $E$ is a non--degenerate curve. We conclude that $\cF$ is stable. 

Let $\cG:=(\pi^*\cF)(h)$: we have $c_1(\cG)=2h$ and it fits into a sequence like \eqref{seq33}. The cohomology of Sequence \eqref{seqRestr} tensorized by $\cE^\vee\otimes\cG$ gives
$$
H^0\big(F,\cE^\vee\otimes\cG\big)\longrightarrow H^0\big(\p2,\cE^\vee\otimes\cG\otimes\cO_{\p2}\big)\longrightarrow H^1\big(F,\cE^\vee\otimes\cG(f-\xi)\big).
$$
Notice that $\cE^\vee\otimes\cG\otimes\cO_{\p2}\cong\cF^\vee\otimes\cF$. 

Tensorizing Sequence \eqref{seq33} for $\cG$ by $\cE^\vee(f-\xi)\cong\cE(-3\xi-f)$, we obtain the exact sequence 
$$
H^1\big(F,\cE(-2\xi-f)\big)^{\oplus4}\longrightarrow H^1\big(\p2,\cE^\vee\otimes\cG(f-\xi)\big)\longrightarrow H^2\big(F,\cE(-2\xi-2f)\big)^{\oplus2}=0.
$$
The cohomology of Sequence \eqref{seqIdeal} twisted by $\cO_F(-2\xi-f)$ gives $h^1\big(F,\cE(-2\xi-f)\big)=h^1\big(F,\cI_{E\vert F}(f)\big)$. The cohomology of Sequence \eqref{seqStandard} twisted by $\cO_F(f)$ and the equality $fE=3$ implies $h^1\big(F,\cI_{E\vert F}(f)\big)=h^0\big(F,\cI_{E\vert F}(f)\big)$, which vanishes because $E$ is non--degenerate. We conclude that $h^1\big(F,\cE^\vee\otimes\cG(f-\xi)\big)=0$.

Thus it is possible to lift the identity morphism $\cF\to\cF$ to a morphism $\psi\colon\cE\to\cG$. We have $c_1(\cG)=2h=c_1$, thus $\det(\psi)\in k$. Since $\psi$ is an isomorphism when restricted to $\p2$ it follows that $\det(\psi)\ne0$, i.e. $\psi$ is a global isomorphism. This completes the proof of the second part of the statement.
\end{proof}

The above discussion shows that the two maps $\cF\mapsto (\pi^*\cF)(h)$ and $\cE\mapsto q^*(\cE(-h))$ are inverse each other. It follows from Proposition \ref{pStable} that the proof of the first part of the statement of the following corollary is immediate. 

\begin{corollary}
\label{c33}
There exists an isomorphism $\cM_F^{s,U}(2;2h,3\xi^2+3f^2)\cong\cM_{\p2}^{s}(2;0,2)$. 

In particular $\cM_F^{s,U}(2;2h,3\xi^2+3f^2)$ is irreducible, smooth, rational of dimension $5$. Moreover 
$$
h^0\big(F,\cE\otimes {\cE}^\vee\big)=1,\quad h^1\big(F,\cE\otimes {\cE}^\vee\big)=5,\quad h^2\big(F,\cE\otimes {\cE}^\vee\big)=h^3\big(F,\cE\otimes {\cE}^\vee\big)=0.
$$
\end{corollary}
\begin{proof}
The first part of the statement has been proved above.  For the properties of $\cM_F^{s,U}(2;2h,3\xi^2+3f^2)$ we refer to \cite{O--S--S}, Theorem II.4.2.1 and, for the assertion on the rationality, to \cite{Ba} (see also \cite{Mar2}, \cite{Ma} and \cite{Ka}).

Recall that stable bundles are simple (see \cite{H--L}, Corollary 1.2.8), thus $h^0\big(F,\cE\otimes {\cE}^\vee\big)=1$. For the two last vanishings see Lemma 2.3 of \cite{C--F--M2}. The value of $h^1\big(F,\cE\otimes {\cE}^\vee\big)$ can be now easily obtained by checking that $\chi(\cE\otimes\cE^\vee)=4-4hc_2+hc_1^2$ using Equality \eqref{RRgeneral} with $c_1({\cE}^\vee\otimes\cE)=0$, $c_2({\cE}^\vee\otimes\cE)=4c_2-c_1^2$, $c_3({\cE}^\vee\otimes\cE)=0$.
\end{proof}

\subsection{The case $4\xi^2+f^2$}
Recall that $\sigma\colon F\to\p3$ denotes the blow up map at the single point $P$. We have $\sigma^*\cO_{\p3}(1)\cong\cO_F(\xi)$. A null--correlation bundle $\cG$ on $\p3$ is the kernel of any surjective morphism $\cT_{\p3}(-1)\to\cO_{\p3}(1)$ where $\cT_{\p3}$ is the tangent bundle. In particular $c_1(\cG)=0$ and $c_2(\cG)$ is the class of a line (conversely each stable bundle of rank $2$ on $\p3$ with such classes is a null--correlation bundle: see \cite{O--S--S}, Lemma 4.3.2). It follows $\cG^\vee\cong\cG$, hence there is an exact sequence
$$
0\longrightarrow\cO_{\p3}(-1)\longrightarrow\Omega^1_{\p3}(1)\longrightarrow\cG\longrightarrow0.
$$
Pulling back the above sequence via $\sigma$ and twisting by $\cO_F(h)$ we finally obtain vector bundles $\cE:=(\sigma^*\cG)(h)$ fitting into the exact sequence 
\begin{equation*}
\label{NullSequence}
0\longrightarrow\cO_F(f)\longrightarrow\sigma^*\Omega^1_{\p3}(2\xi+f)\longrightarrow\cE\longrightarrow0.
\end{equation*}

\begin{proposition}
\label{p41}
There exists Ulrich bundles $\cE$ of rank $2$ on $F$ with $c_1=2h$ and $c_2=4\xi^2+f^2$.

Moreover each such a bundle arises via the construction described above.
\end{proposition}
\begin{proof}
The assertions on the rank and the Chern classes of $\cE$ follow immediately from the construction above.

We have now to show that $\cE$ is aCM and Ulrich, i.e. $h^0\big(F,\cE(-h)\big)=0$ and $h^0\big(F,\cE\big)=14$. Since $\cO_F(f)$ is aCM, it follows that
\begin{equation}
\label{cohomologyE}
\begin{gathered}
h^0\big(F,\cE(-h)\big)=h^0\big(F,\sigma^*\Omega^1_{\p3}(\xi)\big),\\
h^0\big(F,\cE\big)=h^0\big(F,\sigma^*\Omega^1_{\p3}(2\xi+f)\big)-3,\\
h^2\big(F,\cE((-4-t)h)\big)=h^1\big(F,\cE(th)\big)=h^1\big(F,\sigma^*\Omega^1_{\p3}((t+2)\xi+(t+1)f)\big),\qquad t\in\bZ.
\end{gathered}
\end{equation}
It is clear that the statement follows from the computation of the aforementioned cohomology groups of the twists of $\sigma^*\Omega^1_{\p3}(\xi)$. 

By pulling back to $F$ the Euler sequence we obtain
$$
0\longrightarrow\sigma^*\Omega^1_{\p3}(\xi+th)\longrightarrow\cO_F(th)^{\oplus4}\mapright{\varphi_t}\cO_F(\xi+th)\longrightarrow0.
$$
Let $y_0,y_1,y_2,y_3$ be homogeneous coordinates in $\p3$. We recall that the matrix of $\varphi_t$ is
$$
\left(\begin{array}{cccc}s_0&s_1&s_2&s_3\end{array}
\right)
$$
where $s_i:=\sigma^*y_i\in H^0\big(F,\cO_F(\xi)\big)$, $i=0,1,2,3$. Since $\cO_F$ is aCM the cohomology of the above sequence is 
\begin{align*}
0\longrightarrow H^0\big(F,\sigma^*\Omega^1_{\p3}(\xi+th)\big)&\longrightarrow H^0\big(F,\cO_F(th)\big)^{\oplus4}\mapright{\Phi_t}\\
&\longrightarrow H^0\big(F,\cO_F(\xi+th)\big)\longrightarrow H^1\big(F,\sigma^*\Omega^1_{\p3}(\xi+th)\big)\longrightarrow0
\end{align*}
where $\Phi_t$ has still the matrix above.

Via the identification induced by the projection $\pi\colon F\to\p2$, we have
\begin{gather*}
H^0\big(F,\cO_F(th)\big)\cong\bigoplus_{i=0}^t H^0\big(\p2,\cO_{\p2}(i+t)\big),\\
H^0\big(F,\cO_F(\xi+th)\big)\cong\bigoplus_{i=0}^{t+1} H^0\big(\p2,\cO_{\p2}(i+t)\big).
\end{gather*}
The last identification for $t=0$, allows us to assume that  $s_i=x_i$, $i=0,1,2$, are homogeneous coordinates in $\p2$ and $s_3=1$. With all the above identification it is immediate to check that $\Phi_t$ is surjective for each $t$. In particular 
\begin{gather*}
h^0\big(F,\sigma^*\Omega^1_{\p3}(\xi)\big)=0,\\
h^0\big(F,\sigma^*\Omega^1_{\p3}(\xi+h)\big)=17,\\
h^1\big(F,\sigma^*\Omega^1_{\p3}(\xi+th)\big)=0,\qquad t\in\bZ.
\end{gather*}
The proof of the existence part of the statement thus follows by substituting the identities above in Equalities \eqref{cohomologyE}.

Now we turn our attention to prove that if $\cE$ is any Ulrich bundle of rank $2$ on $F$ with $c_1=2h$ and $c_2=4\xi^2+f^2$, then it arises via the construction described above. To this purpose, it suffices to prove  $\cG:=\sigma_*\cE(-h)$ is a null--correlation bundle. Once we have proved such an assertion, then the second part of the statement follows. Indeed consider the natural map of vector bundles $\psi\colon(\sigma^*\cG)(h)\to\cE$. We know that $c_1((\sigma^*\cG)(h))=2h=c_1$, hence the locus where $\psi$ is not an isomorphism is either $F$, or $\emptyset$. Since $\psi$ is certainly an isomorphism at the general point, it follows that it is an isomorphism everywhere.

Our first step for proving that $\cG$ is a null--correlation bundle is to deal with $\cF:=\cE\otimes\cO_{F_1}$. Notice that $hF_1=\xi^2-f^2$ is the class of a line in $F_1\cong\p2$ (see Remark \ref{rdelPezzo}). Thus  $c_1(\cF)=2$ and $c_2(F)=c_2F_1=1$. First we prove that $\cF:=\cE\otimes\cO_{F_1}\cong\cO_{F_1}(1)^{\oplus2}$ by using Horrocks theorem on $F_1$. We have to show that $h^1\big(\p2,\cF(t)\big)=0$ for each $t\in\bZ$. Thanks to Serre's duality, it suffices to prove the vanishing for $t\ge-2$.

In order to prove the above vanishing we again tensorize Sequence \eqref{seqRestr} by $\cE(th)$. Since  $hF_1$ is the class of a line in $\p2$, it follows that the cohomology of the sequence and the vanishing of the intermediate cohomology of $\cE$ yield the equality
\begin{equation}
\label{h^1=h^2}
h^1\big(\p2,\cF(t)\big)=h^2\big(F,\cE((t-1)\xi+(t+1)f)\big).
\end{equation}
The cohomology of Sequence \eqref{seqIdeal} twisted by $\cO_F((t-1)\xi+(t+1)f)$ and Corollary \ref{cNonVanishing} give
$$
h^2\big(F,\cE((t-1)\xi+(t+1)f)\big)=h^2\big(F,\cI_{E\vert F}((t+1)\xi+(t+3)f)\big),\qquad t\ge-1.
$$
The cohomology of Sequence \eqref{seqStandard} twisted by $\cO_F((t+1)\xi+(t+3)f)$ and Corollary \ref{cLineBundle} give
\begin{equation}
\label{Ideal=E}
h^2\big(F,\cI_{E\vert F}((t+1)\xi+(t+3)f)\big)=h^1\big(E,\cO_E((t+1)\xi+(t+3)f)\big),\qquad t\in\bZ.
\end{equation}
We have $((t+1)\xi+(t+3)f)E=9t+17$, thus the last dimension above is zero when $t\ge-1$, because $E$ is a smooth elliptic curve. We conclude that $h^1\big(\p2,\cF(t)\big)=0$ for $t\ge-1$. 

We now check $h^1\big(\p2,\cF(-2)\big)=0$. Equality \eqref{h^1=h^2} gives $h^1\big(\p2,\cF(-2)\big)=h^2\big(F,\cE(-3\xi-f)\big)$. On the one hand,  $(f-\xi)E=-1$, thus Equality \eqref{Ideal=E} gives $h^2\big(F,\cI_{E\vert F}(-\xi+f)\big)=1$. On the other hand, $h^3\big(F,\cE(-3\xi-f)\big)=h^0\big(F,\cE(-\xi-3f)\big)=0$ ($\cE$ is initialized) and $h^2\big(F,\cO(-3\xi-f)\big)=0$ because $\cO_F(2f)$ is aCM (Corollary \ref{cLineBundle}). Thus the cohomology of Sequence \eqref{seqIdeal} gives
$$
0\longrightarrow H^2\big(F,\cE(-3\xi-f)\big)\longrightarrow H^2\big(F,\cI_{E\vert F}(-\xi+f)\big)\longrightarrow H^3\big(F,\cO(-3\xi-f)\big)\longrightarrow0.
$$
Corollary \ref{cNonVanishing} implies $h^3\big(F,\cO(-3\xi-f)\big)=1$, hence $h^1\big(\p2,\cF(-2)\big)=h^2\big(F,\cE(-3\xi-f)\big)=0$.

Thanks to Horrocks theorem on $F_1\cong\p2$ we deduce that $\cF$ splits. Since we have $c_1(\cF)=2$ and $c_2(\cF)=1$ we conclude that $\cF\cong\cO_{F_1}(1)^{\oplus2}$.

Now we show that $R^i\sigma_*\cE=0$ for each $i\ge1$. To this purpose it suffices to check that the stalk $(R^i\sigma_*\cE)_x$ is zero for each $x\in\p3$, which is obvious for $x\ne P$, because $\sigma$ is an isomorphism on $\p3\setminus P$. Thus we have to show that $(R^i\sigma_*\cE)_P=0$. 

Let $\frak m\subseteq\cO:=\cO_{\p3, P}$ be the maximal ideal and set $\cO_n:=\cO/\frak m^n$, $F_n:=F\times_{\p3}\spec(\cO_n)$,
$\cE_n:=\cE\otimes_{\cO_F}\cO_{F_n}$: we have $\cE_1\cong\cF\cong \cO_{F_1}(1)^{\oplus2}$. 

Let $k\cong \cO/\frak m\to\cO$ be the inclusion. Then the composition $\cO/\frak m\to\cO\to\cO/\frak m^n\to\cO/\frak m$ of such an inclusion with the natural projections is the identity. In particular $\cO/\frak m^n$ is an $\cO/\frak m$--algebra, hence we can write $F_n\cong F_1\times_{\spec(\cO_1)}\spec(\cO_n)\cong\mathbb P^2_{\cO_n}$ and
$$
\cE_n\cong\cE\otimes_{\cO_F}(\cO_{F_1}\otimes_{\cO_{1}}\cO_n)\cong\cE_1\otimes_{\cO_F}(\cO_F\otimes_{\cO_1}\cO_n)\cong\cE_1\otimes_{\cO_F}\cO_{F_n}.
$$
Since we know that $\cE_1\cong \cO_{F_1}(1)^{\oplus2}$, it follows that $\cE_n\cong \cO_{F_n}(1)^{\oplus2}$ for each $n\ge1$, hence 
$$
\widehat{(R^i\sigma_*\cE)_P}\cong \varprojlim H^i\big(F_n,\cE_n\big)=0,\qquad i\ge1
$$
(see \cite{Ha2}, Theorem III.11.1), thus $(R^i\sigma_*\cE)_P=0$ too. Exercises III.8.1, III.8.3 of \cite{Ha2} and the isomorphism $\sigma^*\cO_{\p3}(1)\cong\cO_F(\xi)$ imply $h^i\big(\p3,(\sigma_*\cE(-h))(t)\big)=h^i\big(F,\cE((t-1)\xi-f)\big)$, $i\ge0$.

We are now ready to show that $\cG=\sigma_*\cE(-h)$ is a null--correlation bundle. To this purpose we use Beilinson's spectral sequence whose $E_1$ terms are 
$$
E^{p,q}_1:= H^q\big(\p3,\cG(p)\big)\otimes\Omega_{\p3}^{-p}(-p)\cong H^q\big(F,\cE((p-1)\xi-f)\big)\otimes\Omega_{\p3}^{-p}(-p)
$$
(see \cite{O--S--S}, Theorem II.3.1.3 I: though the theorem is stated for vector bundles it holds for every coherent sheaf on $\cO_{\p3}$ as pointed out in \cite{A--O}). 

In the computations below we will repeatedly use of the isomorphism $\cE^\vee\cong\cE(-2h)$, of Corollary \ref{cNonVanishing} and of Proposition \ref{pLineBundle} without mention them explicitly. Moreover we also make use several times of Riemann--Roch theorem on the elliptic curve $E$ for computing the cohomology of the line bundle $\cO_F(\lambda_1\xi+\lambda_2f)$.

We have $h^q\big(F,\cE(-\xi-f)\big)=0$ for $q=1,2$ because $\cE$ is aCM, i.e. $E_1^{0,q}=0$ in the same range. Moreover $h^0\big(F,\cE((p-1)\xi-f)\big)=0$ for each $p\le0$, because $\cE$ is initialized, hence $E_1^{p,0}=0$, $p\le0$. For the same reason $h^3\big(F,\cE(-\xi-f)\big)=h^0\big(F,\cE(-3\xi-3f)\big)=0$, i.e. $E_1^{0,3}=0$.

The cohomology of Sequence \eqref{seqIdeal} twisted by $\cO_F(-2\xi-f)$ implies $h^i\big(F,\cE(-2\xi-f)\big)=h^i\big(F,\cI_{E\vert F}(f)\big)$, $i\ge 1$. The cohomology of Sequence \eqref{seqStandard} twisted by $\cO_F(f)$ and the fact that $fE=4$ give $h^1\big(F,\cI_{E\vert F}(f)\big)=1$ and $h^i\big(F,\cI_{E\vert F}(f)\big)=0$, $i\ge2$. We conclude that $E_1^{-1,1}=\Omega_{\p3}^1(1)$ and $E_1^{-1,2}=E_1^{-1,3}=0$.

By Serre's duality we have $h^3\big(F,\cE(-3\xi-f)\big)=h^0\big(F,\cE(-\xi-3f)\big)=0$ because $\cE$ is initialized, then $E_1^{-2,3}=0$. Thus the cohomology of Sequence \eqref{seqIdeal} twisted by $\cO_F(-3\xi-f)$ gives $h^1\big(F,\cE(-3\xi-f)\big)=h^1\big(F,\cI_{E\vert F}(f-\xi)\big)$ and the existence of an exact sequence of the form
$$
0\longrightarrow H^2\big(F,\cE(-3\xi-f)\big)\longrightarrow H^2\big(F,\cI_{E\vert F}(f-\xi)\big)\longrightarrow k\longrightarrow 0.
$$
The cohomology of Sequence \eqref{seqStandard} twisted by $\cO_F(f-\xi)$ and the equality $(f-\xi)E=-1$ give $h^1\big(F,\cI_{E\vert F}(f-\xi)\big)=0$ and $h^2\big(F,\cI_{E\vert F}(f-\xi)\big)=1$, hence $E_1^{-2,1}=E_1^{-2,2}=0$. We also have $h^3\big(F,\cE(-4\xi-f)\big)=h^0\big(F,\cE(-3f)\big)$. The cohomology of Sequence \eqref{seqIdeal} twisted by $\cO_F(-3f)$ gives $h^0\big(F,\cE(-3f)\big)=h^1\big(F,\cI_{E\vert F}(2\xi-f)\big)$. Every surface in $\vert 2\xi-f\vert$ has degree $5$, thus it is degenerate in $\p8$. Hence the last dimension must be zero because $E$ is a non--degenerate curve, i.e. $E_1^{-3,3}=0$.

The above vanishing show that the unique non--zero differential is $d_2^{-3,2}\colon\cO_{\p3}(-1)^{\oplus c}\to\Omega_{\p3}^1(1)$. Thus $E_\infty^{p,q}$ is non--zero only for $(p,q)=(-1,1)$, and we obtain an exact sequence of the form
$$
0\longrightarrow\cO_{\p3}(-1)^{\oplus c}\mapright{\vartheta}\Omega^1_{\p3}(1)\longrightarrow\cG\longrightarrow0.
$$
We immediately obtain $c=1$, because $\cG$ has rank $2$ at the general point.

We conclude the proof by showing that $\cG$ is a vector bundle. Taking the dual of the above sequence we deduce that $\cG^\vee$ is the kernel of a morphism $\vartheta^{\vee}\colon \cT_{\p3}(-1)\to \cO_{\p3}(1)$. We claim that $\vartheta^\vee$ is surjective, i.e. that the locus of points of $\p3$ where $\vartheta^\vee$ vanishes is empty. Such a locus coincides with the locus $\Theta$ where $\vartheta$ vanishes. Since the restriction of $\sigma$ to $F\setminus F_1$ is an isomorphism onto $\p3\setminus\{\ P\ \}$ it follows that $\Theta$ is supported on a scheme contained in $\{\ P\ \}$. Thanks to \cite{Fu}, Theorem 14.4 b), Example 14.4.1 and easy computations $\deg(\Theta)=c_3(\Omega^1_{\p3}(1)-\cO_{\p3}(-1))=0$, hence $\Theta=\emptyset$. We deduce that $\cG^\vee$ is a vector bundle, thus the bidual of the above sequence yields an isomorphism $\cG\cong(\cG^\vee)^\vee$, i.e. $\cG$ is a null--correlation bundle.
\end{proof}

Thanks to Proposition \ref{pStable} the proof of the first part of the statement of the following corollary is immediate (see Theorem II.4.3.4 of \cite{O--S--S}). The second part can be proved along the same lines of the proof of the analogous part of the statement of Corollary \ref{c33}.

\begin{corollary}
\label{c41}
There exists an isomorphism $\cM_F^{s,U}(2;2h,4\xi^2+f^2)\cong\cM_{\p3}^{s}(2;0,1)$. 

In particular $\cM_F^{s,U}(2;2h,4\xi^2+f^2)$ is irreducible, smooth, rational of dimension $5$. Moreover 
$$
h^0\big(F,\cE\otimes {\cE}^\vee\big)=1,\quad h^1\big(F,\cE\otimes {\cE}^\vee\big)=5,\quad h^2\big(F,\cE\otimes {\cE}^\vee\big)=h^3\big(F,\cE\otimes {\cE}^\vee\big)=0.
$$
\end{corollary}

\subsection{Ulrich--wildness of $F$}
\label{sWild}
We conclude this section by using the above results for proving the Ulrich--wildness of $F$. 
To this purpose we will denote by $\cA$  and $\cB$ two non--isomorphic Ulrich bundles of rank $2$ on $F$: we know that $c_1(\cA)=c_1(\cB)=2h$. 

Recall that $\cB$, being Ulrich has a linear presentation of the form
$$
\cO_{\p 8}(-1)^{\oplus\beta_1}\longrightarrow\cO_{\p 8}^{\oplus\beta_0}\to\cB\longrightarrow0
$$
(e.g. see Proposition 2.1 of \cite{E--S--W}). Twisting the above sequence by $\cO_F$ we obtain an exact sequence of the form
$$
\cO_{F}(-h)^{\oplus\beta_1}\longrightarrow\cO_{F}^{\oplus\beta_0}\to\cB\longrightarrow0.
$$
Indicating by $\mathcal K$ the kernel of the map onto $\cB$, which coincides with the image of $\cO_{F}(-h)^{\oplus\beta_1}\to\cO_{F}^{\oplus\beta_0}$, we finally obtain the exact sequence
$$
0\longrightarrow\mathcal K\longrightarrow\cO_{F}^{\oplus\beta_0}\to\cB\longrightarrow0.
$$
The sheaf $\mathcal K$ is locally free on $F$, because the same is true for both $\cO_{F}^{\oplus\beta_0}$ and $\cB$.

Twisting the above sequence by ${\cA}^\vee$ and taking its cohomology, we obtain
\begin{gather*}
h^3\big(F,\cB\otimes {\cA}^\vee\big)\le \beta_0h^3\big(F,\cA^\vee\big)=\beta_0h^0\big(F,\cA(-2h)\big),\\
h^2\big(F,\cB\otimes {\cA}^\vee\big)\le h^3\big(F,{\mathcal K}\otimes {\cA}^\vee\big)=h^0\big(F,{\mathcal K}^\vee\otimes\cA(-2h)\big),
\end{gather*}
because $\cA$ is aCM and $F$ has dimension $3$. Thus $h^3\big(F,\cB\otimes {\cA}^\vee\big)=0$, because $\cA$ is initialized.

The epimorphism $\cO_{F}(-1)^{\oplus\beta_1}\to\mathcal K$ induces by duality a monomorphism ${\mathcal K}^\vee\otimes\cA(-2h)\to\cA(-h)^{\oplus\beta_1}$. Thus
$$
h^0\big(F,{\mathcal K}^\vee\otimes\cA(-2h)\big)\le \beta_1h^0\big(F,\cA(-h)\big)=0,
$$
again because $\cA$ is initialized.

\begin{lemma}
With the above notation 
\begin{gather*}
h^0\big(F,\cA\otimes\cB^\vee\big)=h^0\big(F,\cB\otimes\cA^\vee\big)=0,\qquad
h^2\big(F,\cB\otimes\cA^\vee\big)=h^3\big(F,\cB\otimes\cA^\vee\big)=0,\\
h^1\big(F,\cB\otimes\cA^\vee\big)=4.
\end{gather*}
\end{lemma}
\begin{proof}
The above discussion shows that $h^2\big(F,\cB\otimes\cA^\vee\big)=h^3\big(F,\cB\otimes\cA^\vee\big)=0$ which is the last vanishing. Moreover 
we know that $\cA$ and $\cB$ are non--isomorphic and stable. Then Proposition 1.2.7 of \cite{H--L} gives the first vanishings.

Equality \eqref{RRgeneral} and the above vanishings yield
$$
h^1\big(F,\cB\otimes\cA^\vee\big)=-\chi(\cB\otimes\cA^\vee)=c_2(\cB\otimes\cA^\vee)h-4
$$
because $c_1(\cB\otimes\cA^\vee)=c_3(\cB\otimes\cA^\vee)=0$. Since $c_2(\cA)h=c_2(\cB)h=9$ and $c_1(\cA)c_1(\cB)h=28$, we finally obtain $h^1\big(F,\cB\otimes\cA^\vee\big)=4$.
\end{proof}

We are now ready to prove the already claimed main result of this section.

\begin{theorem}
The variety $F$ is of Ulrich--wild representation type.
\end{theorem}
\begin{proof}
Recall that $\cA$ and $\cB$ are stable, thus simple by \cite{H--L}, Corollary 1.2.8.
Since $h^1\big(F,\cB\otimes\cA^\vee\big)\geq 3$, the statement follows from Corollary 1 of \cite{F--PL}.
\end{proof}

\section{Indecomposable, aCM bundles which are not Ulrich}
\label{sNonUlrich}
In this last section we will classify all non--Ulrich bundles on $F$. Thus $c_1$ is either $0$,  or the Chern classes of $\cE$ are listed in Table B as A$_i$ with $i=0,\dots,6$ due to Theorem \ref{tDzero}: as pointed out at the end of Remark \ref{rDual}, it suffices to examine, besides the case $c_1=0$, only the cases A$_0$, A$_1$, A$_2$, A$_3$. 

Notice that in cases A$_0$, A$_1$, A$_2$, A$_3$ we have $h^0\big(F,\cE^\vee\big)=h^0\big(F,\cE^\vee(-h)\big)=0$ because $c_1$ is effective and non--zero.

\subsection{The cases A$_0$ and A$_6$}
In case A$_0$ the class of $E$ is $f^2$, hence $E$ is a line which is the pull--back of a point via $\pi$. Conversely, take a line $E$ in the class $f^2$. As above $\det(\cN_{E\vert F})\cong\cO_E$ and the line bundle $\cO_F(f)$ satisfies the hypothesis of Theorem \ref{tSerre}. Thus there is an exact sequence of the form
$$
0\longrightarrow\cO_F\longrightarrow\cE\longrightarrow\cI_{E\vert F}(f)\longrightarrow0.
$$

As in the analysis of the case A$_6$ in the proof of Theorem \ref{tDzero} there is an epimorphism $\cO_F^{\oplus2}\to \cI_{E\vert F}(f)$. Sequence \eqref{seqIdeal} thus gives the existence of an epimorphism $\cO_F^{\oplus3}\to \cE$ whose kernel is necessarily $\cO_F(-f)$. By duality we obtain an exact sequence of the form
\begin{equation}
\label{seqA_0}
0\longrightarrow\cE\longrightarrow\cO_F(f)^{\oplus3}\longrightarrow\cO_F(2f)\longrightarrow0.
\end{equation}
The projection formula gives the existence of an isomorphism
$$
\Hom_F\big(\cO_F(f)^{\oplus3},\cO_F(2f)\big)\cong\Hom_F\big(\cO_{\p2}(1)^{\oplus3},\cO_{\p2}(2)\big):
$$
thus each morphism $\cO_F(f)^{\oplus3}\to\cO_F(2f)$ is the pull--back via $\pi$ of a morphism $\cO_{\p2}(1)^{\oplus3}\to\cO_{\p2}(2)$. We conclude that $\cE\cong (\pi^*\Omega^1_{\p2})(2f)$.

We have that $\cE:=(\pi^*\Omega^1_{\p2})(2f)$ is trivially indecomposable. Pulling back via $\pi$ the Euler sequence we obtain an exact sequence as in \eqref{seqA_0}. 
Since the line bundles $\cO_F(f)$ and $\cO_F(2f)$ are aCM, the twisted cohomology of Sequence \eqref{seqA_0} yields $h^2\big(F,\cE(th)\big)=0$, $t\in\bZ$. Twisting the dual of Sequence \eqref{seqA_0} by $\cO_F(f)$ and taking its twisted cohomology we also obtain that $h^1\big(F,\cE(th)\big)=0$, $t\in\bZ$, because $\cO_F$ is aCM too. We conclude that $\cE$ is aCM. Finally again the cohomology of Sequence \eqref{seqA_0} yields $h^0\big(F,\cE\big)\ne0$ and $h^0\big(F,\cE(-h)\big)=0$, i.e. $\cE$ is initialized.

We conclude that both cases A$_0$ and A$_6$ occur. More precisely in case A$_0$ we have $\cE\cong(\pi^*\Omega^1_{\p2})(2f)$, while in case A$_6$ we have $\cE\cong(\pi^*\Omega^1_{\p2})(\xi+2f)$. In both the cases $\cE$ is globally generated, thus each general section of $\cE$ vanishes exactly along a smooth curve $E$. In case A$_0$, $E$ is a line. In case A$_6$, $E$ is a rational quintic.

\subsection{The cases A$_1$, A$_2$, A$_4$ and A$_5$}
In case A$_1$ the class of $E$ is $\xi^2-f^2$, hence $E$ is again a line, hence $\det(\cN_{E\vert F})\cong\cO_E$. Since $\xi(\xi^2-f^2)=0$, it follows that $\cO_F(\xi)\otimes \cO_E\cong\cO_E$.  Thus the vanishing $h^1\big(F,\cO_F(-\xi)\big)=h^2\big(F,\cO_F(-\xi)\big)=0$ and Theorem \ref{tSerre} imply the existence of a unique exact sequence of the form
$$
0\longrightarrow\cO_F\longrightarrow\cE\longrightarrow\cI_{E\vert F}(\xi)\longrightarrow0.
$$
Confronting such a sequence with Sequence \eqref{xi^2-f^2} twisted by $\cO_F(\xi)$, we conclude that $\cE\cong \cO_F(\xi-f)\oplus\cO_F(f)$.Similarly, in case A$_2$ the class of $E$ is $f^2$, again $E$ is a line and $\det(\cN_{E\vert F})\cong\cO_E$. As in the previous case we can prove that $\cE\cong \cO_F(f)^{\oplus2}$.

We conclude that the cases A$_1$, A$_2$, A$_4$ and A$_5$ do not occur, because the bundle is always decomposable.

\subsection{The case A$_3$}
The class of $E$ is $\beta\xi^2+2(1-\beta)f^2$ so that $\deg(E)=2$. We know (see the proof of Theorem \ref{tDzero}) that the class of $E$ is one of the following: $2f^2$, $\xi^2$, $2(\xi^2-f^2)$. 

There is a unique exact sequence of the form
$$
0\longrightarrow\cO_F\longrightarrow\cE\longrightarrow\cI_{E\vert F}(h)\longrightarrow0
$$
because $h^1\big(F,\cO_F(-h)\big)=h^2\big(F,\cO_F(-h)\big)=0$ (see Theorem \ref{tSerre}).

In the case the class of $E$ is $2f^2$, then $fE=0$, hence the cohomology of Sequence \eqref{seqStandard} twisted by $\cO_F(f)$ yields $h^0\big(F,\cI_{E\vert F}(f)\big)\ge h^0\big(F,\cO_{F}(f)\big)-h^0\big(E,\cO_{E}(f)\big)=2$. We deduce the existence of two distinct divisors $A,B\in\vert f\vert$ through $E$. The divisors $A$ and $B$ have no common components, thus $A\cap B$ is a curve of degree $1$ containing $E$, which has degree $2$, a contradiction.

Let us examine the two remaining cases. If $E$ is in the class $\xi^2$, then we can find divisors $A\in\vert f\vert$ and $B\in\vert \xi\vert$ through $E$. Again an easy computation shows that $A\cap B$ is a curve containing $E$. Since its degree is $2$, it follows that $E=A\cap B$. As in the cases A$_1$ and A$_2$, the uniqueness of $\cE$ implies $\cE\cong\cO_F(f)\oplus\cO_F(\xi)$. 

If the class of $E$ is $\xi^2-f^2$, then we can find divisors $A\in\vert\xi-f\vert$ and $B\in\vert 2f\vert$ such that $E=A\cap B$ and we can again repeat the above argument.

We conclude that also the case A$_3$ does not occur, because the bundle is decomposable too.

\subsection{The case $c_1=0$}
In this case we have an exact sequence of the form
\begin{equation}
\label{seqc_1=0}
0\longrightarrow\cO_F\longrightarrow\cE\longrightarrow\cI_{E\vert F}\longrightarrow0.
\end{equation}
Thus $h^0\big(F,\cE^\vee\big)=1$: Equality \eqref{hc_2} yields that the zero locus of a non--zero section of $\cE$ is a line. We conclude that $c_2$ is either $f^2$, or $\xi^2-f^2$, thanks to Remark \ref{rdelPezzo}. 

In what follows we will give two interesting constructions for an indecomposable, initialized aCM bundle $\cE$ with $c_1=0$. 

\begin{example}
\label{eP2Dual}
We have 
$$
\dim\big(\Ext^1_F\big(\cO_F(\xi-f),\cO_F(-\xi+f)\big)\big)=h^1\big(F,\cO_F(-2\xi+2f)\big)=3.
$$
Thus there exist exact sequences of the form
\begin{equation}
\label{seqExtension}
0\longrightarrow \cO_F(f-\xi)\longrightarrow\cE\longrightarrow\cO_F(\xi-f)\longrightarrow0.
\end{equation}
The line bundes $\cO_F(2f)$, hence $\cO_F(f-\xi)$, and $\cO_F(\xi-f)$ are aCM (see Corollary \ref{cLineBundle}). Hence $\cE$ is aCM and $h^0\big(F,\cE\big)=1$. A simple computation also shows that it is initialized, $c_1=0$ and $c_2=\xi^2-f^2$. In particular the zero--locus of a non--zero section of $\cE$ is a line $E$ in the class $\xi^2-f^2$. 

We claim that the bundle $\cE$ is indecomposable. Indeed if it is decomposable, then there are $\lambda_1,\lambda_2\in\bZ$ such that
$$
(\lambda_1\xi+\lambda_2 f)(-\lambda_1\xi-\lambda_2f)=\xi^2-f^2.
$$
Simple computations show that either $\cE\cong\cO_F(\xi-f)\oplus\cO_F(f-\xi)$, or $\cE\cong\cO_F(\xi+f)\oplus\cO_F(-\xi-f)$. 

In the second case the $\cE$ would not be initialized, a contradiction. Thus only the first case is admissible and the same argument as before shows that the extension itself would split, again a contradiction.
\end{example}

\begin{example}
\label{eP2}
Take a general morphism $\cO_F(f)^{\oplus2}\oplus\cO_F\to\cO_F(2f)$. Its cokernel is supported on a zero--dimensional scheme $Z\subseteq F$ by standard Bertini theorem. Thanks to \cite{Fu}, Theorem 14.4 b), Example 14.4.1 we see that $\deg(Z)=c_3(\cO_F(f)^{\oplus2}\oplus\cO_F-\cO_F(2f))=0$, hence $Z=\emptyset$. Thus we obtain a vector bundle $\cE$ fitting into an exact sequence of the form
\begin{equation}
\label{seqPullBack}
0\longrightarrow \cO_F(-2f)\longrightarrow\cO_F(-f)^{\oplus2}\oplus\cO_F\longrightarrow\cE\longrightarrow0.
\end{equation}

The line bundles $\cO_F(\xi)$ and $\cO_F(\xi-f)$ are aCM (see Corollary \ref{cLineBundle}), hence the same is true for $\cO_F(-f)$ and $\cO_F(-2f)$. In particular the bundle $\cE$ fitting in the above sequence satisfies $h^1\big(F,\cE(th)\big)=0$, $t\in\bZ$. Since $\cE^\vee\cong\cE$, it follows from these vanishings and Serre duality that $h^2\big(F,\cE(th)\big)=0$, $t\in\bZ$. We deduce that $\cE$ is aCM. Moreover, A simple computation also shows that it is initialized, $c_1=0$ and $c_2=f^2$. 

Finally, $\cE$ is indecomposable because the equality
$$
(\lambda_1\xi+\lambda_2 f)(-\lambda_1\xi-\lambda_2f)=f^2
$$
has no solutions in $\bZ$. 
\end{example}

The two above examples prove the following statement.

\begin{proposition}
\label{pc_1=0}
There exists initialized, indecomposable, aCM bundles $\cE$ of rank $2$ on $F$ with $c_1=0$ and $c_2$ either $f^2$, or $\xi^2-f^2$.

Moreover each such a bundle arises via the construction described above.
\end{proposition}
\begin{proof}
The existence part of the above statement follows from Example \ref{eP2Dual} and \ref{eP2}. 

Now let $E$ be a line on $F$, so that its class in $A(F)$ is either $f^2$, or $\xi^2-f^2$. By adjunction on $F$ we know that $\cO_{E}(-2h)\cong\cO_{\p1}(-2)\cong \cO_F(-2h)\otimes\det(\cN_{E\vert F})$, thus $\det(\cN_{E\vert F})\cong\cO_E$. The line bundle $\cO_F$ satisfies the hypothesis of Theorem \ref{tSerre}, thus there is a vector bundle $\cE$ of rank $2$ on $F$ fitting into Sequence \eqref{seqc_1=0}.

Let $\xi^2-f^2$ be the class of $E$ (hence it is also $c_2$). As pointed out in Remark \ref{rdelPezzo}, the line $E$ is contained in the unique divisor $F_1\in\vert\xi-f\vert$. It follows that $h^0\big(F,\cI_{E\vert F}(\xi-f)\big)\ge1$ and equality must hold because, otherwise, the class of $E$ in $A(F)$ would be $f^2-\xi^2$, a contradiction.

The cohomology of Sequence \eqref{seqc_1=0} twisted by $\cO_F(\xi-f)$ thus gives $h^0\big(F,\cE(\xi-f)\big)=2$. Let $\sigma$ be a general section of  $\cE(\xi-f)$. Then $(\sigma)_0=\Gamma\cup\Delta$ where $\Gamma$ and $\Delta$ have pure dimensions $1$ and $2$ respectively. Thus Sequence \eqref{seqIdeal} for $\sigma$ becomes
\begin{equation}
\label{seqGammaDelta}
0\longrightarrow \cO_F(\Delta)\longrightarrow\cE(\xi-f)\longrightarrow\cI_{\Gamma\vert F}(2\xi-2f-\Delta)\longrightarrow0.
\end{equation}

Let $\cO_F(\Delta)=\cO_F(\delta_1\xi+\delta_2f)$. The inequality $1\le h^0\big(F,\cO_F(\Delta)\big)\le h^0\big(F,\cE(\xi-f)\big)=2$ and Proposition \ref{pLineBundle} imply $\delta_1+\delta_2=0$, thus $\cO_F(\Delta)=\cO_F(\delta(\xi-f))$ for some integer $\delta\ge0$. In particular $h^0\big(F,\cO_F(\Delta)\big)=1$, thanks to the aforementioned proposition. Moreover
$$
1=h^0\big(F,\cE(\xi-f)\big)-h^0\big(F,\cO_F(\Delta)\big)\le h^0\big(F,\cI_{\Gamma\vert F}(2\xi-2f-\Delta)\big).
$$
It follows that $\delta\le1$. If $\delta=1$, then Sequence \eqref{seqGammaDelta} coincides with Sequence \eqref{seqc_1=0} twisted by $\cO_F(\xi-f)$. In this case Sequence \eqref{seqGammaDelta} is exactly Sequence \eqref{seqc_1=0} twisted by $\cO_F(\xi-f)$, hence $\sigma$ is in the image of the morphism
$$
H^0\big(F,\cE\big)\otimes H^0\big(F,\cO_F(\xi-f)\big)\to H^0\big(F,\cE(\xi-f)\big)
$$
which trivially has dimension $1$. Since $\sigma$ is general, it follows that $\delta=0$, i.e. $\Delta$ is empty. Since the class of $\Gamma$ is $c_2(\cE(\xi-f))=0$, we conclude that Sequence \eqref{seqGammaDelta} twisted by $\cO_F(f-\xi)$ is exactly Sequence \eqref{seqExtension}.

Let $f^2$ be the class of $E$ (hence it is also $c_2$). As pointed out in Remark \ref{rdelPezzo}, the line $E$ is the inverse image of a single point $e\in\p2$. Theorem \ref{tSerre} yields the existence of a vector bundle $\cF$ on $\p2$ fitting into an exact sequence of the form
$$
0\longrightarrow \cO_{\p2}\longrightarrow\cF\longrightarrow\cI_{e\vert\p2}\longrightarrow0.
$$
Moreover, we trivially have a surjective map $\cO_{\p2}(-1)^{\oplus2}\to\cI_{e\vert\p2}$. By applying $\Hom_{\p2}\big(\cO_{\p2}(-1)^{\oplus2},\cdot\big)$ to the above sequence and taking into account that
$$
\Ext^1_{\p2}\big(\cO_{\p2}(-1)^{\oplus2},\cO_{\p2}\big)=H^1\big({\p2},\cO_{\p2}(1)\big)^{\oplus2}=0
$$
we deduce that the surjection $\cO_{\p2}(-1)^{\oplus2}\to\cI_{e\vert {\p2}}$ lifts to a map $\cO_{\p2}(-1)^{\oplus2}\to\cF$. We thus obtain a surjective morphism $\cO_{\p2}(-1)^{\oplus2}\oplus\cO_{\p2}\to\cF$. A Chern class computation shows that the kernel of this map is $\cO_{\p2}(-2)$, hence we have an exact sequence 
$$
0\longrightarrow \cO_{\p2}(-2)\longrightarrow\cO_{\p2}(-1)^{\oplus2}\oplus\cO_{\p2}\longrightarrow\cF\longrightarrow0.
$$
By pulling back the above sequence via $\pi$, which is a flat morphism, we deduce the existence of Sequence \eqref{seqPullBack}
\end{proof}

We conclude this section with two remarks.

\begin{remark}
\label{rPullBack}
It is interesting to notice that all the bundles which we have dealt with can be obtained, up to twists, as pull--back of suitable bundles via $\sigma$ or $\pi$. This is no longer true for the bundles with $c_1=0$, but for the case $c_2=f^2$.

Indeed, assume the existence of a vector bundle $\cG$ on $\p3$ such that $\cE\cong\sigma^*\cG(\lambda_1\xi+\lambda_1f)$. Assume that $c_1(\sigma^*\cG)=\sigma^*c_1(\cG)=g_1\xi$ and $c_2(\sigma^*\cG)=\sigma^*c_2(\cG)=g_2\xi^2$. We would have
$$
0=c_1(\cE)=c_1(\sigma^*\cG(\lambda_1\xi+\lambda_2f))=(g_1+2\lambda_1)\xi+2\lambda_2f
$$ 
so we obtain $\lambda_2=0$. Thus
$$
c_2(\cE)=c_2(\sigma^*\cG(\lambda_1\xi+\lambda_2f))=(g_2+\lambda_1g_1+\lambda_1^2)\xi^2,
$$
which cannot be either $\xi^2-f^2$, or $f^2$. Thus $\cE$ is not a pull--back of a suitable bundle via $\sigma$. 

Similar computations, with $\pi$ and $f$ instead of $\sigma$ and $\xi$, show that,  when $c_2=\xi^2-f^2$, the bundle $\cE$ is not a pull--back of a  bundle via $\pi$, as well.

When $c_2=f^2$ the picture is different. Indeed the construction given in Example \ref{eP2} shows that $\cE\cong\pi^*\cF$ where $\cF$ is a vector bundle on $\p2$ with $c_1(\cF)=0$ associated via the Hartshorne--Serre correspondence to a point.
\end{remark}

\begin{remark}
\label{rStable}
Repeating verbatim the proof of Proposition 3.3 of \cite{C--F--M2} one can easily prove that each indecomposable, initialized, aCM bundle of rank $2$ with $c_1=0$ on $\cE$ is not semistable (but it is $\mu$--semistable). 

Notice that the above Examples \ref{eP2Dual}, \ref{eP2} and Proposition \ref{pc_1=0} show that these bundles are fibres of a flat family whose base is the dijoint union of two copies of $\p2$.
\end{remark}

\bigskip
\noindent
Gianfranco Casnati,\\
Dipartimento di Scienze Matematiche, Politecnico di Torino,\\
c.so Duca degli Abruzzi 24,\\
10129 Torino, Italy\\
e-mail: {\tt gianfranco.casnati@polito.it}

\bigskip
\noindent
Matej Filip,\\
Institut f\"ur Mathematik, Freie Universit\"at Berlin,\\
Arnimallee 3 Raum 120,\\
14195 Berlin, Germany\\
e-mail: {\tt filip@math.fu-berlin.de}

\bigskip
\noindent
Francesco Malaspina,\\
Dipartimento di Scienze Matematiche, Politecnico di Torino,\\
c.so Duca degli Abruzzi 24,\\
10129 Torino, Italy\\
e-mail: {\tt francesco.malaspina@polito.it}

\end{document}